\documentclass[12pt]{amsart}

\usepackage{CJKutf8}
\usepackage{amssymb}
\usepackage{eucal}
\usepackage{color}

\usepackage{multirow,bigdelim}
\usepackage{amsfonts}
\usepackage{dsfont}
\usepackage{amsmath}
\usepackage{mathrsfs}
\usepackage{todonotes}
\usepackage{hyperref}
\usepackage{multimedia}
\usepackage{graphicx}
\usepackage{indentfirst}
\usepackage{bm}
\usepackage{amsthm}
\usepackage{mathrsfs}
\usepackage{latexsym}
\usepackage{amsmath}
\usepackage{amssymb}
\usepackage{microtype}
\usepackage{relsize}
\usepackage{hyperref}
\usepackage[numbers,sort&compress]{natbib}

\DeclareSymbolFont{SY}{U}{psy}{m}{n}
\DeclareMathSymbol{\emptyset}{\mathord}{SY}{'306}

\theoremstyle{plain}

\newtheorem{thm}{Theorem}[section]
\newtheorem{cor}[thm]{Corollary}
\newtheorem{lem}[thm]{Lemma}
\newtheorem{prop}[thm]{Proposition}
\theoremstyle{definition}
\newtheorem{defn}[thm]{Definition}
\newtheorem{rem}[thm]{Remark}
\newtheorem{ex}[thm]{Example}

\numberwithin{equation}{section}

\def\K{\mathcal K}

\def\e{equivalent}

\def\beq{\begin{eqnarray}}
\def\eeq{\end{eqnarray}}
\def\beqa{\begin{eqnarray*}}
\def\eeqa{\end{eqnarray*}}

\def\T{\boldsymbol T}

\begin{document}
\title{On the irreducibility and weakly homogeneity of a class of operators}

\author{Shanshan Ji and Xiaomeng Wei$^*$}
\curraddr[S. Ji, X. Wei]{School of Mathematical Sciences, Hebei Normal University,
Shijiazhuang, Hebei 050016, China}

\email[S. Ji]{jishanshan15@outlook.com}
\email[X. Wei]{Weixiaomeng1209@163.com}

\thanks{This work was supported by National Natural Science Foundation of China (Grant No. 11831006 and 11922108).}
\thanks{* Corresponding author}
\begin{abstract}
To construct more homogeneous operators, B. Bagchi and G. Misra in \cite{d} introduced the operator $\left(\begin{smallmatrix}
T_0 & T_0-T_1 \\
0 & T_1\\
\end{smallmatrix}\right)$ and proved that when $T_0$ and $T_1$ are homogeneous operators with the same unitary representation $U(g)$, it is homogeneous with associated representation $U(g)\oplus U(g)$.
At the same time, they asked an open question, is the constructed operator irreducible?
A. Kor$\acute{a}$nyi in \cite{e} showed that when the (1,2)-entry of the matrix is $\alpha(T_0-T_1)$, $\alpha\in\mathbb{C}$ the above result is also valid, and their unitary equivalence class depends only on $|\alpha|$.
In this case, he and S. Hazra \cite{f} gave a large class of irreducible homogeneous bilateral $2\times2$ block shifts, respectively, which are mutually unitarily inequivalent for $\alpha>0$.
In this note, we generalize the construction to $T=\left(\begin{smallmatrix}
T_0 & XT_1-T_0X \\
0 & T_1\\
\end{smallmatrix}\right)$ and provide some sufficient conditions for its irreducibility.
We also find that for the above-mentioned $T_0,T_1$ and non-scalar operator $X$, $T$ is weakly homogeneous rather than homogeneous. So the weak homogeneity problem related to $T$ is investigated.
\end{abstract}

\maketitle

\section{Introduction}

Let $\mathcal{H}$ be a complex separable Hilbert space and $\mathcal{L}(\mathcal{H})$ denote the collection of bounded linear operators on $\mathcal{H}$.
Let $T_{i}\in\mathcal{L}(\mathcal{H}_{i})$, $i=0,1$. We say $T_{0}$ is similar to $T_{1}\,\,(T_{0}\sim_{s}T_{1})$ if there exists an invertible operator $X:\,\mathcal{H}_{0}\to\mathcal{H}_{1}$ such that $XT_{0}=T_{1}X$. In particular, when $X$ is a unitary operator, we say $T_{0}$ is  unitarily equivalent to $T_{1}\,\,(T_{0}\sim_{u}T_{1})$.

In the study of operator theory, it has been an active topic to find a model for the class of operators satisfying certain conditions.
When dim $\mathcal{H}<\infty$, the Jordan canonical form theorem provides a model for all operators in $\mathcal{L}(\mathcal{H})$ up to similarity.
When unitary equivalence is considered, it is found by spectral theory that every normal matrix can be written as an orthogonal direct sum of scalar operators on the eigenspace.
In the case of dim $\mathcal{H}=\infty$, the orthogonal direct sum will be replaced by a continuous direct sum or direct integration.
However, there is still no general solution for operators whose spectra are not thin, or for direct sums of these operators.
For example, the backward shift operator $U_+^*: l^2(\mathbb{N})\rightarrow l^2(\mathbb{N})$, defined by
$U_+^*(\alpha_0, \alpha_1, \alpha_2, \cdots)=(\alpha_1, \alpha_2, \alpha_3, \cdots)$. That means $\mathbb{D}\subseteq\sigma(U_+^*)$, the spectrum of $U_+^*$.

In 1978, M.J. Cowen and R.G. Douglas systematically studied a subclass $\mathcal{B}_n(\Omega)$ of $\mathcal{L}(\mathcal{H})$ in \cite{a}, which contain $U_+^*$ and their point spectra contain a bounded connected open set $\Omega$ of $\mathbb{C}$.
There are many criteria for local equivalence to determine the unitary equivalence of operators, including Hermitian holomorphic vector bundles, connections, the curvature, local operators, etc in \cite{a, CD2,ChenD, JJHOU}.
Operators in $\mathcal{B}_n(\Omega)$ correspond to the adjoints of the multiplication operators on the reproducing kernel Hilbert space, and the unitary equivalence of them can also be fully described using the reproducing kernel due to \cite{CS, a, zhu}.
Similar transformations break rigidity, so there are not many geometrically similar equivalent results for such operators. Recently, Cowen-Douglas operators with certain structures have been studied \cite{DKT,n,j,JJKX,JKSX,JJK,JJW}.

We know that $U_+^*\in\mathcal{B}_1(\mathbb{D})$ and its curvature is $\mathcal{K}_{U_+^*}(w)=-\frac{1}{(1-|w|^2)^2}$.
G. Misra in \cite{b} proved that if $T\in\mathcal{B}_1(\mathbb{D})$ is contractive, then $\mathcal{K}_{T}(w)\leq\mathcal{K}_{U_+^*}(w)$, $w\in\mathbb{D}$.
For such an operator $T$, R.G. Douglas once asked if there is a point $w_0\in\mathbb{D}$ such that $\mathcal{K}_{T}(w_0)=-\frac{1}{(1-|w_0|^2)^2}$, is it inferred that $T\sim_uU_+^*$?
Subsequently, G. Misra \cite{la} gave a characterization of contractive homogeneous operators in $\mathcal{B}_{1}(\mathbb{D})$
and indicated that the question raised by R.G. Douglas is positive for homogeneous operators, but not always correct for other operators.
Some affirmative answers are also given in \cite{MR}.

Irreducible homogeneous operators in the Cowen-Douglas class have attracted much attention, while a huge theory has been built.
Many impressive results have been obtained using the representation theory of Lie groups, complex geometry and matrix structures, including explicit representations of these operators and unitary classification (see \cite{020,BG,BM3,BM,JJKM,AM,KM0,KM1,hoM}).
These operators are not only rich in structure, but in most cases also serve as models for other operators.
In \cite{c}, all homogeneous shift operators are listed due to B. Bagchi and G. Misra, it follows that there are many homogeneous operators outside of Cowen-Douglas operators.
Then some attempts were made to construct more matrix-valued homogeneous operators
and several open questions in this research direction were listed, one of them is whether the constructed operator $\left(\begin{smallmatrix}
T_0 & T_0-T_1 \\
0 & T_1\\
\end{smallmatrix}\right)$ is irreducible, where $T_0$ and $T_1$ are homogeneous operators with the same unitary representation (see \cite{d}).
Generalizing the above operator to $\left(\begin{smallmatrix}
T_0 & \alpha(T_0-T_1) \\
0 & T_1\\
\end{smallmatrix}\right)$ for $\alpha\in\mathbb{C}$, A. Kor$\acute{a}$nyi in \cite{e} proved their unitary classification only depends on $|\alpha|$.
In this form of operators, A. Kor$\acute{a}$nyi
first described a class of irreducible homogeneous bilateral block shifts, where $T_0$ and $T_1$ are shifts with weight sequences $\{(\frac{n+a}{n+b})^{\frac{1}{2}}\}_{n\in\mathbb{Z}}$ and $\{(\frac{n+b}{n+a})^{\frac{1}{2}}\}_{n\in\mathbb{Z}}$, $0<a,b<1$ and $a\neq b$.
Afterward, S. Hazra \cite{f} also found two classes of irreducible homogeneous bilateral 2-by-2 shifts, one of which has the structure shown above.
It was also shown that these three classes of operators are mutually unitarily inequivalent when the parameter $\alpha>0$.
From this, a new operator structure is generalized as follows:
\begin{equation}\label{202310.191}
T=\left(\begin{smallmatrix}
T_0 & XT_1-T_0X \\
0 & T_1\\
\end{smallmatrix}\right)
\end{equation}
is equipped with a non-zero bounded linear operator $X$.
It is known that irreducible operators are minimal operator units modulo unitary equivalence. Inspired by the above results, we consider when $T$ in (\ref{202310.191}) is irreducible.

Homogeneous operators have yielded numerous results. As its extension, weakly homogeneous operators are defined (see \cite{d,DNCM}).
In \cite{i}, S. Ghara pointed out the necessary and sufficient conditions for the multiplication by the coordinate function $z$ acting on a Hilbert space $\mathcal{H}_K$ possessing a sharp reproducing kernel $K$, to be weakly homogeneous, and gave some nontrivial results on weakly homogeneous, so that many high index weakly homogeneous operators were obtained.
So far, the characterization of weakly homogeneous operators remains a challenging problem, even for operators of index one in the Cowen-Douglas class.
Because $T$ in (\ref{202310.191}) is strongly reducible, its homogeneity has expired. Thus, it is natural to investigate their weak homogeneity.

The paper is organized as follows.
In \S{2}, we introduced some preliminaries. 
In \S{3}, we discussed when $T$ in (\ref{202310.191}) is irreducible and provide some sufficient conditions, which is related to the open question 11 proposed by B. Bagchi and G. Misra in \cite{d}.
In \S{4}, some weakly homogeneous problems of $T$ in (\ref{202310.191}) were investigated.

\section{Preliminaries}

\subsection{Cowen-Douglas operators}

M.J. Cowen and R.G. Douglas in \cite{a} introduced a class of operators $\mathcal{B}_n(\Omega)$, named Cowen-Douglas operators, defined as follows:

\begin{defn} \cite{a}
For a connected open subset $\Omega$ of $\mathbb{C}$ and a positive integer $n$, let

$$\begin{array}{lll}
\mathcal{B}_n(\Omega):=\{T\in\mathcal{L}(\mathcal{H})|&(1)\,\,\Omega\subset\sigma(T),\\
&(2)\,\,\mbox{ran}(T-w)=\mathcal{H}\,\,\text{for\,all}\,\,w\in\Omega,\\
&(3)\,\,\bigvee_{w\in\Omega}\mbox{ker}(T-w)=\mathcal{H},\\
&(4)\,\,\mbox{dim}\,\mbox{ker}(T-w)=n\,\,\text{for\, all}\,\,w\in\Omega\},
\end{array}$$
where $\sigma(T)$ is the spectrum of the operator $T$ and $\bigvee$ stands for the closed linear span.
\end{defn}

If $T$ is in $\mathcal{B}_{n}(\Omega)$, then it is possible to choose $n$ linearly independent eigenvectors in $\mbox{ker}(T-w)$, which are holomorphic as functions. Thus $w\mapsto\mbox{ker}(T-w)$ defines a rank $n$ Hermitian holomorphic vector bundle $E_{T}$ over $\Omega$,
$$E_{T}=\{(w,x)\in\Omega\times\mathcal{H}:x\in\mbox{ker}(T-w)\}$$ with the natural projection map $\pi:E_{T}\to\Omega,\pi(w,x)=w$.
Let $\{\gamma_{i}\}^{n}_{i=1}$ be a holomorphic frame of $E_{T}$. Then $\{\gamma_{i}(w)\}^{n}_{i=1}$ is a basis of $\mbox{ker}(T-w)$ and the metric of $E_{T}$ is the Gram matrix $h(w)=(\langle\gamma_{j}(w),\gamma_{i}(w)\rangle)^{n}_{i,j=1}$.
The curvature is given by $\mathcal{K}_{T}(w)=-\frac{\partial}{\partial\overline{{w}}}[h^{-1}(w)\frac{\partial}{\partial{w}}h(w)],w\in\Omega$. In fact, the curvature depends on the choice of the frame, except when $n=1$.
In the case of $E_T$ is a line bundle, its curvature is $\mathcal{K}_{T}(w)=-\frac{\partial^{2}}{\partial{w}\overline{\partial}w}\mbox{log}\|\gamma(w)\|^{2}$, where $\gamma$ is a non-zero holomorphic section of $E_T$.

A function $K:\Omega\times\Omega\rightarrow \mathcal{M}_n(\mathbb{C})$ is said to be a non-negative definite kernel if for any subset $(w_1,\cdots,w_k)$ in $\Omega$, the $k\times k$ block matrix
$\big(K(w_i,w_j)\big)_{i,j=1}^k$ is non-negative definite, that is, $\sum_{i,j=1}^k\langle K(w_i,w_j)\eta_j,\eta_i\rangle\geq0,\eta_1,\cdots,\eta_n\in\mathbb{C}^n$.
We say that $K(\cdot,\cdot)$ is sesqui-analytic, if $K$ is holomorphic for the first variable and anti holomorphic for the second variable.
For a sesqui-analytic non-negative definite kernel $K:\Omega\times\Omega\rightarrow \mathcal{M}_n(\mathbb{C})$, a well-known theorem given by Moore that there exists a unique Hilbert space $\mathcal{H}_K$ consisting of $\mathbb{C}^n$-valued
holomorphic functions $E_w^*\xi=K(\cdot,w)\xi$ on $\Omega$ such that the evaluation map $E_w$ is bounded for each $w\in\Omega$ and $K$ is the reproducing kernel of
$\mathcal{H}_K$, where the reproducing property is $\langle f,K(\cdot,w)\xi\rangle_{\mathcal{H}_K}=\langle f(w),\xi\rangle_{\mathbb{C}^n}$, $\xi\in\mathbb{C}^n$.
Based on this, for $T\in\mathcal{B}_n(\Omega)$, $T$ can be realized as the adjoint of the multiplication operator $M_z$ on a reproducing kernel Hilbert space $\mathcal{H}_K$ of holomorphic $\mathbb{C}^n$-valued functions on $\Omega^*$ with reproducing kernel $K$, denoted by $T\sim_u(M_z^*,\mathcal{H}_K)$, and $M_z^*K(\cdot,\bar{w})\xi=wK(\cdot,\bar{w})\xi$ for $w\in\Omega$ and $\xi\in\mathbb{C}^n$.
The reader is referred to \cite{CS,a, zhu}.

\subsection{Homogeneous operators}\label{2.3who}

Let $M\ddot{o}b$ denote the $M\ddot{o}bious$ group of all biholomorphism automorphisms of $\mathbb{D}:=\{z\in\mathbb{C}:|z|<1\}$. Then the form of $\phi$ in $M\ddot{o}b$ is $\phi(z)=e^{i\theta}\frac{z-a}{1-\overline{a}z},\theta\in\mathbb{R},a\in\mathbb{D}$.
Operator $T\in\mathcal{L}(\mathcal{H})$ is said to be homogeneous if $\sigma(T)\subset\overline{\mathbb{D}}$, the closed unit disc and $\phi(T)\sim_{u}T$ for every $\phi$ in $M\ddot{o}b$ (see \cite{020,la,hoM}).

Let $\mathcal{U}(\mathcal{H})$ be the group of unitary operators on $\mathcal{H}$. A Borel function $\pi:\,G\to\mathcal{U}(\mathcal{H})$ is said to be a projective presentation of $G$ on the Hilbert space $\mathcal{H}$.
If there is a projective representation $\pi$ of $M\ddot{o}b$ on $\mathcal{H}$ with the property $\phi(T)=\pi(\phi)^{*}T\pi(\phi),$
then $\pi$ is said to be the representation associated with the operator $T$.
For every irreducible homogeneous operator, there exists a unique (up to equivalence) projective representation associated with it (see \cite{c}).

A complete characterization of the contractive homogeneous operator in $\mathcal{B}_{1}(\mathbb{D})$ is pointed out by G. Misra in \cite{la} to be that its curvature is equal to $-\lambda(1-|w|^{2})^{-2}$ for some positive real number $\lambda$.
That means $T$ and $T^{*}$ are homogeneous, if $T\sim_{u}(M^{*}_{z},\mathcal{H}_{K^{(\lambda)}})$,  $K^{(\lambda)}(z,w)=\frac{1}{(1-z\overline{w})^{\lambda}}$, $z,w\in\mathbb{D}$.
Besides, there are three bilateral shift operators that are also homogeneous operators: the unweighted bilateral shift $B$, $K_{a,b}$ with weight sequence $\sqrt{\frac{n+a}{n+b}},n\in\mathbb{Z}$ for any two distinct real numbers $a$ and $b$ in $(0,1)$, and $B_{x}$ with weight sequence $\cdots,1,x,1,\cdots$ ($x$ in the zeroth slot, $1$ elsewhere). The above list is exhaustive.

As a generalization of homogeneous operators, weakly homogeneous operators are defined in \cite{d,DNCM}.
Let $T\in\mathcal{L}(\mathcal{H})$, the operator $T$ is said to be a weakly homogeneous operator if $\sigma(T)\subseteq\bar{\mathbb{D}}$ and for each $\phi\in M\ddot{o}b$, $\phi(T)\sim_{s}T$.
For $M_z$ acting on $\mathcal{H}_K$, the boundedness of the composition operator $C_\phi$, defined by $C_\phi(f)=f\circ\phi, f\in\mathcal{H}_K$ means the weak homogeneity of $M_z$.
For any $\gamma\in\mathbb{R}$, let $K_{(\gamma)}$ be the positive definite kernel given by
$K_{(\gamma)}(z,w)=\sum\limits_{n=0}^\infty(n+1)^\gamma z^n\bar{w}^n$, $z,w\in\mathbb{D}$.
N. Zorboska\cite{g}, C.C. Cowen and B.D. MacCluer\cite{h} proved that $C_\phi$ is bounded on $\mathcal{H}_{K^{(\lambda)}}$, $\lambda>0$ and $\mathcal{H}_{K_{(\gamma)}}$, $\gamma\in\mathbb{R}$, that is, $(M_z^*,\mathcal{H}_{K^{(\lambda)}})$ and
$(M_z^*,\mathcal{H}_{K_{(\gamma)}})$ are weakly homogeneous.

\subsection{Irreducible operator}

Let $T\in\mathcal{L}(\mathcal{H})$. The commutant algebra of $T$ is $\mathcal{A}'(T):=\{X|XT=TX\}$.
For a closed linear subspace $\mathcal{M}$ of $\mathcal{H}$, $\mathcal{M}$ is said to be a reduced subspace of $T$ if both $\mathcal{M}$ and $\mathcal{M}^{\bot}$ are invariant subspaces of $T$.
$T\in\mathcal{L}(\mathcal{H})$ is said to be irreducible if $T$ only has trivial reduced subspaces, equivalently, $\mathcal{A}'(T)$ only has trivial orthogonal projections in \cite{Halmos}. $T$ is strongly irreducible, if $\mathcal{A}'(T)$ does not contain idempotent operators other than 0 and $I$.
In what follows, $T\in(SI)$ means $T$ is strongly irreducible. 

C.L. Jiang and his collaborators introduced the $K_0$-group of the operator commutant algebra into the study of operator theory and provided a similar classification theorem of Cowen-Douglas operators (see \cite{q,01,o}).
Operator $T\in\mathcal{L}(\mathcal{H})$ has a finite (SI) decomposition if and only if there exist $k\in\mathbb{N}$, $n_{1},\cdots,n_{k}\in\mathbb{N}$ and $T_{1},\cdots,T_{k}\in(SI)$ such that $T_{i}\nsim_{s}T_{j}(i\neq j)$ and $T=T^{(n_{1})}_{1}\oplus\cdots\oplus T^{(n_{k})}_{k}$ (see \cite{q}).

\begin{defn}\cite{q}
Let $T\in\mathcal{L}(\mathcal{H})$ have a finite (SI) decomposition $T=T^{(n_{1})}_{1}\oplus \cdots\oplus T^{(n_{k})}_{k}$. We say $T$ has a unique (SI) decomposition up to similarity if there exists $S$ such that $S\sim_{s}T$ and $S$ has a finite (SI) decomposition $S=S^{(m_{1})}_{1}\oplus\cdots\oplus S^{(m_{l})}_{l}$, then $k=l$ and
there exists a permutation $\pi$ of the set $\{1,2,\cdots,k\}$ such that $T^{(n_{i})}_{i}\sim_{s}S^{(m_{\pi(i)})}_{\pi(i)}$ and $n_{i}=m_{\pi(i)}$.
\end{defn}

\begin{thm}\cite{o}\label{220506.1}
Let $A,B\in\mathcal{B}_{n}(\mathbb{D})$. Suppose that $A^{(n_{1})}_{1}\oplus\cdots\oplus A^{(n_{k})}_{k}$ is a finite (SI) decomposition of $A$.
Then $A\sim_{s}B$ if and only if:
\begin{enumerate}
  \item $(K_{0}(\mathcal{A}'(A\oplus B)),\bigvee(\mathcal{A}'(A\oplus B)),I)\cong(\mathbb{Z}^{(k)},\mathbb{N}^{(k)},1)$;
  \item The isomorphism $h$ from $\bigvee(\mathcal{A}'(A\oplus B))$ to $\mathbb{N}^{(k)}$ sends $[I]$ to $(2n_{1},2n_{2},\cdots,2n_{k})$, i.e., $h([I])=2n_{1}e_{1}+2n_{2}e_{2}+\cdots+2n_{k}e_{k}$, where $I$ is the unit of $\mathcal{A}'(A\oplus B)$ and $\{e_{i}\}^{k}_{i=1}$ are the generators of $\mathbb{N}^{(k)}$.
\end{enumerate}
\end{thm}

\begin{thm}\cite{01}\label{220506.2}
Two strongly irreducible Cowen-Douglas operators $A$ and $B$ are similar if and only if there exists a positive integer $n$ such that $A^{(n)}\sim_{s}B^{(n)}$.
\end{thm}

\section{Irreducibility of weighted block shift operators}

Let $T\in\mathcal{L}(\mathcal{H})$. Then $T$ is called a block shift if there is an orthogonal decomposition $\mathcal{H}=\oplus_{n\in I}W_{n}$ of $\mathcal{H}$ with non-trivial subspaces $W_{n},n\in I$, such that $T(W_{n})\subseteq W_{n+1}$ for all $n$ in $I$.
$T$ is said to be a bilateral or unilateral block shift if $I=\mathbb{Z},\mathbb{Z}^{+}$ or $\mathbb{Z}^{-}$. The subspaces $W_{n},n\in I$ are called the blocks of $T$.
When dim $W_{n}=1$, $T$ is called a shift operator.
In this case, if $W_{n}=span\{e_n\}$, then $\{e_n\}_{n\in I}$ is the orthonormal basis of $\mathcal{H}$.
A. Kor$\acute{a}$nyi and S. Hazra found three large classes irreducible bilateral block shifts in \cite{e,f}.
B. Bagchi and G. Misra proved that any irreducible homogeneous operator is a block shift in \cite{c}.
Motivated by these results, in this section we consider the irreducibility of certain block shifts of block size to be two.
Some sufficient conditions are given.

Before stating our results, we first specify some common notations.
Let $l^{2}(\mathbb{Z})=\{(a_{i})_{i\in\mathbb{Z}}:\sum_{i\in\mathbb{Z}}|a_{i}|^{2}<\infty\}$ and $T_{0},T_{1}$ on $l^{2}(\mathbb{Z})$ be bilateral weighted shifts with  non-zero weighted sequences $\{w_{n}\}_{n\in\mathbb{Z}}$ and $\{v_{n}\}_{n\in\mathbb{Z}}$ respect to orthonormal basis$\{e_n\}_{n\in \mathbb{Z}}$ of $l^{2}(\mathbb{Z})$.
Define $T(w_{n},v_{n},d_{n}):=\left(\begin{smallmatrix}
T_0 & XT_1-T_0X \\
0 & T_1\\
\end{smallmatrix}\right)$ acting on $\mathcal{H}=l^{2}(\mathbb{Z})\oplus l^{2}(\mathbb{Z})$, where $X$ satisfies $Xe_{n}=d_{n}e_{n}$ for a uniformly bounded sequence $\{d_n\}_{n\in\mathbb{Z}}$ and $XT_1\neq T_0X$.
When $w_{n}=\frac{1}{v_{n}}=t_{n}$, $T(w_{n},v_{n},d_{n})=T(t_{n},d_{n})$.
In these forms, let $\mathcal{B}_n=\{\left (\begin{smallmatrix}e_{n}\\0\\\end{smallmatrix}\right ),\left (\begin{smallmatrix}0\\e_{n}\\\end{smallmatrix}\right )\}$ be a orthonormal set and $H(n)=\mbox{span}\mathcal{B}_n$.
Write $T_n=T|_{H(n)}$, we can see $A_{n}:=T^{*}_{n}T_{n}$ and $B_{n}:=T_{n-1}T^{*}_{n-1}$ map $H(n)$ to $H(n)$.
With these notations, we have the following results.

\subsection{The irreducibility of $T(t_{n},d_{n})$}

\begin{prop}\label{220507.1}
Let $T=T(t_n,\alpha)$ be a bounded operator for non-zero complex number $\alpha$.
If
 \begin{enumerate}
 \item [(1)]
there exists $i_0\in\mathbb{Z}$ such that $t^{2}_{n}+\frac{1}{t^{2}_{n}}=t^{2}_{m}+\frac{1}{t^{2}_{m}}$ only holds when $m=-(n+i_0)$;
 \item [(2)]
$t_{n}>0$ and $t_n\neq\frac{1}{t_{n-1}}\neq1$,
\end{enumerate}
then $T$ is irreducible.
\end{prop}

\begin{proof}
Let $\mathcal{K}$ be an arbitrary non-zero reduced subspace of $T$. We will check that $\mathcal{K}=\mathcal{H}=span\{\left (\begin{smallmatrix}e_{n}\\0\\\end{smallmatrix}\right ),\left (\begin{smallmatrix}0\\e_{n}\\\end{smallmatrix}\right ),n\in\mathbb{Z}\}$.
Since $\mathcal{K}$ is a closed linear subspace and $H(n)=\mbox{span}\{\left (\begin{smallmatrix}e_{n}\\0\\\end{smallmatrix}\right ),\left (\begin{smallmatrix}0\\e_{n}\\\end{smallmatrix}\right )\}$, if we prove that for each $n$, $H (n)\subset\mathcal{K}$, then the theorem is proven.
In the following, we will complete the proof in three steps.

\textbf{Step 1}: There exists $n_0\in\mathbb{Z}$ such that $\mathcal{K}\cap H(n_0)\neq\emptyset$.

If $\mathcal{K}$ is a reducing subspace of $T$, then $T^{*}T\mathcal{K}\subset T^{*}\mathcal{K}\subset\mathcal{K}, T^{*}T\mathcal{K}^{\perp}\subset T^{*}\mathcal{K}^{\perp}\subset\mathcal{K}^{\perp}$, $\mathcal{K}$ is an invariant subspace of both $T^{*}T$ and $TT^{*}$.
From the spectral theorem of positive operators, we can conclude that the restrictions of elements in $\mathcal{K}$ on the eigenspaces of $T^*T$ and $TT^*$ still belong to $\mathcal{K}$.
Next, we will calculate the eigenspaces of $T^*T$ and $TT^*$ by rearranging the orthonormal basis of $\mathcal{H}$ and partitioning the operators.

It is easy to see that $T$ maps $H(n)$ to $H(n+1)$ and
$T_n=\left(\begin{smallmatrix}t_{n}&\alpha(\frac{1}{t_{n}}-t_{n})\\0&\frac{1}{t_{n}}\\\end{smallmatrix}\right )$.
It follows that
$$A_{n}
=\left(\begin{smallmatrix}t^{2}_{n}&\alpha(1-t^{2}_{n})\\ \bar{\alpha}(1-t^{2}_{n})&|\alpha|^{2}(\frac{1}{t_{n}}-t_{n})^{2}+\frac{1}{t^{2}_{n}}\\\end{smallmatrix}\right ),
B_{n}
=\left(\begin{smallmatrix}t^{2}_{n-1}+|\alpha|^{2}(\frac{1}{t_{n-1}}-t_{n-1})^2&\alpha(\frac{1}{t^{2}_{n-1}}-1)\\
\bar{\alpha}(\frac{1}{t^{2}_{n-1}}-1)&\frac{1}{t^{2}_{n-1}}\\\end{smallmatrix}\right ).$$
We have $A_{n}\neq I$ and its determinant is $1$.
Since the spectrum of positive operators is not-negative, the eigenvalues of $A_{n}$ must be real numbers.
So there exists a real number $\lambda_{n}>1$ such that the eigenvalues of $A_{n}$ are $\lambda^{2}_{n}$ and $\frac{1}{\lambda^{2}_{n}}$.
Note that $\mbox{trace}A_{n}=(|\alpha|^{2}+1)(t^{2}_{n}+\frac{1}{t^{2}_{n}})-2|\alpha|^{2}=\lambda^{2}_{n}+\frac{1}{\lambda^{2}_{n}}$.
By using the condition $t^{2}_{n}+\frac{1}{t^{2}_{n}}=t^{2}_{-(n+i_0)}+\frac{1}{t^{2}_{-(n+i_0)}}$ for some $i_0\in\mathbb{Z}$.
This is equivalent to
$\lambda^{2}_{n}+\frac{1}{\lambda^{2}_{n}}=\lambda^{2}_{-(n+i_0)}+\frac{1}{\lambda^{2}_{-(n+i_0)}}$ for some $i_0\in\mathbb{Z}$.
Since $\lambda_{n}>1$ for any $n\in\mathbb{Z}$, that is $\lambda_{n}=\lambda_{-(n+i_0)}$ for some $i_0\in\mathbb{Z}$.
Let $\lambda^{(1)}_{n}=\lambda^{2}_{n}$ and $\lambda^{(2)}_{n}=\frac{1}{\lambda^{2}_{n}}$.
Suppose that $\{v^{(1)}_{n},v^{(2)}_{n}\}$ is the orthonormal basis which makes $A_{n}$ diagonal.
Let $u^{(k)}_{n}=T_{n-1}v^{(k)}_{n-1}$. That means $B_{n}u^{(k)}_{n}=T_{n-1}T^{*}_{n-1} T_{n-1}v^{(k)}_{n-1}=\lambda^{(k)}_{n-1}T_{n-1}v^{(k)}_{n-1}=\lambda^{(k)}_{n-1}u^{(k)}_{n}$.
Since the orthogonality of $v^{(1)}_{n}$ and $v^{(2)}_{n}$, $u^{(1)}_{n}$ and $u^{(2)}_{n}$ are also,
$\langle u^{(1)}_{n}, u^{(2)}_{n}\rangle=\langle T^{*}_{n-1}T_{n-1}v^{(1)}_{n-1}, v^{(2)}_{n-1}\rangle=\langle\lambda_{n-1}v^{(1)}_{n-1}, v^{(2)}_{n-1}\rangle=0$.
Clearly, $\{u^{(1)}_{n}, u^{(2)}_{n}\}$ is an orthonromal basis of $H(n)$ which makes $B_{n}$ diagonal.

I will complete the proof of Step 1 in two cases.

\textbf{Case 1}: $i_0$ is an odd number. Then the multiplicity of every eigenvalues of $T^{*}T$ is two.
Let $\mathcal{A}_{n,k}=\mbox{span}\{v^{(k)}_{n}, v^{(k)}_{-(n+i_0)}\},\ k=1,2$ be the eigenspace of $T^{*}T$ corresponds to the eigenvalue $\lambda_{n}^{(k)}=\lambda_{-(n+i_0)}^{(k)}$. Then $\mathcal{H}=\bigoplus_{k=1,2;n\geq\frac{1-i_0}{2}}\mathcal{A}_{n,k}$.
For a fixed but arbitrary a non-zero vector $f\in\mathcal{K}$, we have $f=\sum_{k=1,2;n\geq\frac{1-i_0}{2}}\alpha_{n,k}h_{n,k}$, where $h_{n,k}\in\mathcal{A}_{n,k}$.
Since $f$ is non-zero, there exist $n,k$ such that $\alpha_{n,k}h_{n,k}\neq0$. It implies that $h_{n,k}\neq0$ and  $h_{n,k}\in\mathcal{K}$, since $h_{n,k}$ is the restriction of $f$ on the eigenspace $\mathcal{A}_{n,k}$.
Then there exist $\gamma,\delta$ that are not all zero such that $h_{n,k}=\gamma v^{(k)}_{n}+\delta v^{(k)}_{-(n+i_0)}$, where $v^{(k)}_{n}\in H(n)$ and $v^{(k)}_{-(n+i_0)}\in H(-(n+i_0))$.
Furthermore, for any $m\in\mathbb{Z}$, we have $T^{m}h_{n,k}=\hat{\delta}\hat{h}_{m-n-i_0}+\hat{\gamma}\hat{h}_{m+n}\in\mathcal{K}$, where $\hat{h}_{m-n-i_0}\in H(m-n-i_0), \hat{h}_{m+n}\in H(m+n)$.
From the arbitrariness of $m$, we can assume that $m$ is sufficiently large, up to $m-n-i_0\geq\frac{1-i_0}{2}, m+n\geq\frac{1-i_0}{2}$.
In this case, the eigenvalues of the $T^*T$ restricted to the spaces $H(m-n-i_0)$ and $H(m+n)$ must not be the same.
Then there exist $\delta_{1},\delta_{2}, \gamma_{1}, \gamma_{2}$ that are not all zero such that  $T^{m}h_{n,k}=\hat{\delta}\delta_{1}v^{(1)}_{m-n-i_0}+\hat{\delta}\delta_{2}v^{(2)}_{m-n-i_0}+\hat{\gamma}\gamma_{1}v^{(1)}_{m+n}+\hat{\gamma}\gamma_{2}v^{(2)}_{m+n}$.
Since every operator $A_n=T_n^*T_n$ is invertible and $h_{n,k}\neq0$, so is $T^{m}h_{n,k}\neq0$.
We know that $v^{(1)}_{m-n-i_0}, v^{(2)}_{m-n-i_0}, v^{(1)}_{m+n}, v^{(2)}_{m+n}$ are the eigenvectors corresponding to different eigenvalues of $T^{*}T$.
Therefore at least one of the coefficients of the vectors $v^{(1)}_{m-n-i_0}, v^{(2)}_{m-n-i_0}, v^{(1)}_{m+n}, v^{(2)}_{m+n}$ is not zero. This implies that at least one of the vectors $v^{(1)}_{m-n-i_0}, v^{(2)}_{m-n-i_0}, v^{(1)}_{m+n}, v^{(2)}_{m+n}$ belongs to $\mathcal{K}$. It follows that $\mathcal{K}\cap H(n_0)\neq\emptyset$ for some $n_0$.

\textbf{Case 2}: $i_0$ is an even number. Then the multiplicity of the eigenvalue $\lambda_{-\frac{i_0}{2}}$ of $T^{*}T$ is one, the multiplicity of other eigenvalues is two.
Let $\mathcal{H}=\bigoplus_{k=1,2; n\geq-\frac{i_0}{2}}\mathcal{A}_{n,k}$, where  $\mathcal{A}_{-\frac{i_0}{2},k}=\mbox{span}\{v^{(k)}_{-\frac{i_0}{2}}\}$ is the eigenspace of $T^{*}T$ with the eigenvalue $\lambda_{-\frac{i_0}{2}}$
and $\mathcal{A}_{n,k}=\mbox{span}\{v^{(k)}_{n}, v^{(k)}_{-(n+i_0)}\}$ is the eigenspace of $T^{*}T$ with the eigenvalue $\lambda_{n}=\lambda_{-(n+i_0)}$ for $n>-\frac{i_0}{2}$.
For any non-zero vector $f$ in $\mathcal{K}$, we have $f=\sum_{k=1,2; n\geq-\frac{i_0}{2}}\alpha_{n,k}h_{n,k}$, where $h_{n,k}\in\mathcal{A}_{n,k}$.
Similarly, there exist $n\geq-\frac{i_0}{2}$ and $k=1$ or $2$ such that $\alpha_{n,k}h_{n,k}\neq0$.
If $n=-\frac{i_0}{2}$, then $\mathcal{K}\cap H(-\frac{i_0}{2})\neq\emptyset$. The proof of other situations is similar to Case 1.

To sum up, these ensure the existence of $n_0\in\mathbb{Z}$ such that $\mathcal{K}\cap H(n_0)\neq\emptyset$. This completes the proof of step 1.

\textbf{Step 2}: For any $n\in\mathbb{Z}$, $\mathcal{K}\cap H(n)\neq\emptyset$.

From Step 1, we know that $\mathcal{K}\cap H(n_0)\neq\emptyset$ for some $n_0\in\mathbb{Z}$.
Since $T_n$ is invertible for any $n\in\mathbb{Z}$, if $g\in\mathcal{H}$ is a non-zero function, so is $Tg$.
By using $\mathcal{K}$ is a reduced subspace and applying $T^m$ and $T^{*m}$ to $H(n_{0})\cap\mathcal{K}$ for all $m\in\mathbb{N}$,
we have that $\mathcal{K}\cap H(n_0\pm m)\neq\emptyset$.
By the arbitrariness of $m$, we have completed the proof of this step.

\textbf{Step 3}: $\mathcal{K}=\mathcal{H}$.

Let $h_{n}\in\mathcal{K}\cap H(n)$ and $h_{n}\neq0$. Then there are $\alpha$ and $\beta$ that are not all zero, as well as $\gamma$ and $\delta$ that make $h_{n}=\alpha v^{(1)}_{n}+\beta v^{(2)}_{n}=\gamma u^{(1)}_{n}+\delta u^{(2)}_{n}$.

\textbf{Claim}: At least one of $\alpha\beta$ and $\gamma\delta$ is not zero.

Otherwise, if both $\alpha\beta$ and $\gamma\delta$ are zero, then there is only one of $\alpha$ and $\beta$ is zero, and $\gamma$ and $\delta$ are also. That is to say,
there is a multiple relationship between $\{v^{(1)}_{n}, v^{(2)}_{n}\}$ and $\{u^{(1)}_{n}, u^{(2)}_{n}\}$.
Without loss of generality, we suppose that there exists non-zero $c\in\mathbb{C}$ such that $u^{(1)}_{n}=cv^{(1)}_{n}$.
Then we can find $c_{1}, c_{2}\in\mathbb{C}$ that are not all zero and satisfy $u^{(2)}_{n}=c_{1}v^{(1)}_{n}+c_{2}v^{(2)}_{n}$.
Taking the inner product of $v^{(1)}_{n}$ with $u^{(2)}_{n}$, we obtain
$\langle u^{(2)}_{n}, v^{(1)}_{n}\rangle=\langle c_{1}v^{(1)}_{n}+c_{2}v^{(2)}_{n}, v^{(1)}_{n}\rangle=c_{1}\| v^{(1)}_{n}\|^{2}.$
On the other hand, $\langle u^{(2)}_{n}, v^{(1)}_{n}\rangle=\langle u^{(2)}_{n}, \frac{1}{c}u^{(1)}_{n}\rangle=0$.
So $c_{1}=0$ and $u^{(2)}_{n}=c_{2}v^{(2)}_{n}$. In this case, we infer that
$(A_{n}B_{n}-B_{n}A_{n})u^{(1)}_{n}=c\lambda^{(1)}_{n-1}A_{n}v^{(1)}_{n}-c\lambda^{(1)}_{n}B_{n}v^{(1)}_{n}=c\lambda^{(1)}_{n-1}\lambda^{(1)}_{n}v^{(1)}_{n}-\lambda^{(1)}_{n}B_{n}u^{(1)}_{n}=0.$
Similarly, after replacing $u^{(1)}_{n}$ with $u^{(2)}_{n}$, $v^{(1)}_{n}$ and $v^{(2)}_{n}$ are all zero, respectively.
Thus $A_{n}B_{n}-B_{n}A_{n}=0$.

By a routine computation, we obtain
\begin{align*}
& \quad\,\
A_{n}B_{n}=(B_{n}A_{n})^*\\
&=\left(\begin{smallmatrix}t_n^2t_{n-1}^2+|\alpha|^2\big(t_n^2(\frac{1}{t_{n-1}}-t_{n-1})^2+(1-t_n^2)(\frac{1}{t_{n-1}^2}-1)\big)&\alpha\big(\frac{1}{t_{n-1}^2}-t_n^2\big)\\
\bar{\alpha}(1+|\alpha|^2)\big(t_{n-1}^2-t_{n}^2t_{n-1}^2+\frac{1}{t_{n}^2t_{n-1}^2}-\frac{1}{t_{n}^2}\big)+\bar{\alpha}|\alpha|^2\big(t_{n}^2-\frac{1}{t_{n-1}^2}\big)&|\alpha|^2\big((\frac{1}{t_{n-1}^2}-1)(1-t_{n}^2)+\frac{1}{t_{n-1}^2}(\frac{1}{t_{n}}-t_{n})^2\big)+\frac{1}{t_{n}^2t_{n-1}^2}\\\end{smallmatrix}\right),
\end{align*}
then the (1,2)-entry and (2,1)-entry in $A_{n}B_{n}-B_{n}A_{n}$ are  $\bar{\alpha}(1+|\alpha|^2)\big((t^{2}_{n}-1)(t^{2}_{n-1}-1)-(\frac{1}{t^{2}_{n}}-1)(\frac{1}{t^{2}_{n-1}}-1)\big)$, and other entries are zero.
Since $\alpha\neq0$, $A_{n}B_{n}-B_{n}A_{n}=0$ if and only if
$(t^{2}_{n}-1)(t^{2}_{n-1}-1)=(\frac{1}{t^{2}_{n}}-1)(\frac{1}{t^{2}_{n-1}}-1)=\frac{1}{t^{2}_{n}t^{2}_{n-1}}(1-t^{2}_{n})(1-t^{2}_{n-1})$.
Based on condition (2), this is impossible. So our assumption is invalid.
This implies that the Claim is checked.

If $\alpha\beta\neq0$, we can obtain $v^{(1)}_{n}\in\mathcal{K}\cap H(n)$ by restricting $h_{n}$ to the eigensubspace corresponding to the eigenvalue $\lambda^{(1)}_{n}$ of $T^{*}T$.
In the same way, we can also obtain $v^{(2)}_{n}\in\mathcal{K}\cap H(n)$. So $v^{(1)}_{n}, v^{(2)}_{n}\in\mathcal{K}$. Similarly,  $\gamma\delta\neq0$ means that $u^{(1)}_{n}, u^{(2)}_{n}\in\mathcal{K}$.
Since both $\{v^{(1)}_{n},v^{(2)}_{n}\}$ and $\{u^{(1)}_{n},u^{(2)}_{n}\}$ are orthonormal bases of $H(n)$, we conclude $H(n)\subset\mathcal{K}$.
It follows from Step 2 that $\mathcal{K}=span H(n)=\mathcal{H}$.

Hence, $T$ is irreducible.
\end{proof}

\begin{rem}
Let $T=T(t_n,\alpha)$ and $\tilde{T}=T(\tilde{t}_n,\tilde{\alpha})$. Suppose that $T,\tilde{T}$ satisfy the conditions of Proposition \ref{220507.1},
that is, they satisfy $\lambda_{n}=\lambda_{-(n+i_0)}$ and $\tilde{\lambda}_{n}=\tilde{\lambda}_{-(n+\tilde{i}_0)}$, respectively.
Recall that if $i_0$ is odd, then the multiplicity of the eigenvalues is two. If $i_0$ is an even number, then the multiplicity of the eigenvalue $\lambda_{-\frac{i_0}{2}}$ is one and the multiplicity of other eigenvalues is two. These also apply to $\tilde{\lambda}_{n}$.
We know that $T\sim_s\tilde{T}$ if and only if singular values and their multiplicities of $T$ and $\tilde{T}$ are the same.
For $i_0$ and $\tilde{i}_0$, if one of them is odd and the other is even, $T$ and $\tilde{T}$ must not be similar.
In the following, we will only consider when they are odd numbers.

Then the necessary and sufficient condition of $T\sim_s\tilde{T}$ is
$\left\{\lambda_{n},\frac{1}{\lambda_{n}}\right\}_{n\geq\frac{1-i_0}{2}}=\left\{\tilde\lambda_{n},\frac{1}{\tilde\lambda_{n}}\right\}_{n\geq \frac{1-\tilde{i}_0}{2}}$.
This is equivalent to $\{\lambda_{n}\}_{n\geq\frac{1-i_0}{2}}=\{\tilde\lambda_{n}\}_{n\geq \frac{1-\tilde{i}_0}{2}}$ or
$\left\{\lambda^{2}_{n}+\frac{1}{\lambda^{2}_{n}}\right\}_{n\geq\frac{1-i_0}{2}}=\left\{\tilde\lambda^{2}_{n}+\frac{1}{\tilde\lambda^{2}_{n}}\right\}_{n\geq \frac{1-\tilde{i}_0}{2}}$. That is,
$\left\{(|\alpha|^{2}+1)\left(t_{n}^{2}+\frac{1}{t^{2}_{n}}\right)-2|\alpha|^{2}\right\}_{n\geq\frac{1-i_0}{2}}=\left\{(|\tilde\alpha|^{2}+1)\left(\tilde t_{n}^{2}+\frac{1}{\tilde t^{2}_{n}}\right)-2|\tilde\alpha|^{2}\right\}_{n\geq \frac{1-\tilde{i}_0}{2}}$.
The set composed of singular values is the completely similar invariant of operator $T(t_n,\alpha)$.
Moreover, when $|\alpha|=|\tilde\alpha|$, $T\sim_s\tilde{T}$ if and only if
$\{t_{n}^{2}+\frac{1}{t^{2}_{n}}\}_{n\geq\frac{1-i_0}{2}}=\{\tilde t^{2}_{n}+\frac{1}{\tilde t^{2}_{n}}\}_{n\geq \frac{1-\tilde{i}_0}{2}}$.
\end{rem}

For Proposition \ref{220507.1}, we give an example.
\begin{ex}
Given two non-zero sequences $\{x_{n}\}_{n\in\mathbb{Z}}$ and $\{y_{n}\}_{n\in\mathbb{Z}}$.
Let $T=T(\sqrt{|\frac{x_{n}}{y_{n}}|},\alpha)$ for nonzero $\alpha\in\mathbb{C}$.
Suppose that $x_ny_{n-1}\neq x_{n-1}y_n$ and $x_n\neq y_n$.
If $x_{n}=-y_{-(n+1)}$ for all $n\in\mathbb{Z}$, then $T$ is irreducible.
\end{ex}

\begin{proof}
Since $x_{n}=-y_{-(n+1)}$, then $x_{-(n+1)}=-y_{-(-n-1+1)}=-y_{n}$, $y_{n}=-x_{-(n+1)}$. We notice that $\frac{x_{n}}{y_{n}}=\frac{-y_{-(n+1)}}{-x_{-(n+1)}}=\frac{y_{-(n+1)}}{x_{-(n+1)}}$. Then $\frac{x_{n}}{y_{n}}+\frac{y_{n}}{x_{n}}=\frac{x_{-(n+1)}}{y_{-(n+1)}}+\frac{y_{-(n+1)}}{x_{-(n+1)}}$.

Let $t^{2}_{n}=|\frac{x_{n}}{y_{n}}|$.
Then we need to show that $t_{n}$ satisfies the conditions of Proposition \ref{220507.1}.
By a simple calculation, it follows that
$$t^{2}_{n}+\frac{1}{t^{2}_{n}}=\left|\frac{x_{n}}{y_{n}}\right|+\left|\frac{y_{n}}{x_{n}}\right|=\left|\frac{x_{-(n+1)}}{y_{-(n+1)}}\right|+\left|\frac{y_{-(n+1)}}{x_{-(n+1)}}\right|=t^{2}_{-(n+1)}+\frac{1}{t^{2}_{-(n+1)}}.$$
By assumptions of sequences $\{x_{n}\}_{n\in\mathbb{Z}}$ and $\{y_{n}\}_{n\in\mathbb{Z}}$, we have $t_n\neq t_{n-1}$ and $t_n\neq1$.
This implies that $T$ is irreducible.
\end{proof}

When $x_{n}=n+a,y_{n}=n+b$ and $0<a\neq b<1$, $a+b=1$,
the above example corresponds to A. Kor$\acute{a}$nyi's Theorem 2.2 in \cite{e}. Our results formally generalize his example from the perspective of irreducibility, but the operators we discuss do not have homogeneity, even weak homogeneity.

In the following, we consider the case where the weight sequence of the operator is complex, and we also make some generalization of the operator structure.

\begin{thm}\label{220507.2}
Let $T=T(t_n,d_n)$.
Suppose that there exists a bijective mapping $g$ on $\mathbb{Z}$ and $N\in\mathbb{Z}$ such that when $n>N$,
the sets of singular values of $A_n$ are distinct from each other.
If $trace\,A_n=trace\,A_{g(n)}$ 
and one of the following two conditions is true,
 \begin{enumerate}
 \item [(1)]
$\frac{\overline{t_{n}}}{t_{n-1}}c_{n}\overline{c_{n-1}}\notin\mathbb{R}$, where $c_{n}=\frac{d_{n+1}}{t_{n}}-d_{n}t_{n}$;
 \item [(2)]
$t_{n}\overline{ c_{n}}(|t_{n-1}|^{2}+\frac{1}{|t_{n-1}|^{2}}+|c_{n-1}|^{2})\neq\frac{1}{t_{n-1}}\overline{ c_{n-1}}(|t_{n}|^{2}-\frac{1}{|t_{n}|^{2}}-|c_{n}|^{2})$,
\end{enumerate}
then T is irreducible.
\end{thm}

\begin{proof}
Since $T_{n}=\left(\begin{smallmatrix}t_{n}&\frac{d_{n+1}}{t_{n}}-d_{n}t_{n}\\0&\frac{1}{t_{n}}\\\end{smallmatrix}\right)$,
\begin{align*}A_{n}&=T^{*}_{n}T_{n}
=\left(\begin{smallmatrix}|t_{n}|^{2}&\frac{d_{n+1}\overline{ t_{n}}}{t_{n}}-d_{n}|t_{n}|^{2}\\\frac{\overline{d_{n+1}}t_{n}}{\overline{t_{n}}}-\overline{ d_{n}}|t_{n}|^{2}&|\frac{d_{n+1}}{t_{n}}-d_{n}t_{n}|^{2}+\frac{1}{|t_{n}|^{2}}\\\end{smallmatrix}\right),\\
B_{n}&=T_{n-1}T^{*}_{n-1}
=\left(\begin{smallmatrix}|t_{n-1}|^{2}+|\frac{d_{n}}{t_{n-1}}-d_{n-1}t_{n-1}|^{2}&\frac{d_{n}}{|t_{n-1}|^{2}}-\frac{d_{n-1}t_{n-1}}{\overline{ t_{n-1}}}\\\frac{\overline{d_{n}}}{|t_{n-1}|^{2}}-\frac{\overline{d_{n-1}}\overline{ t_{n-1}}}{t_{n-1}}&\frac{1}{|t_{n-1}|^{2}}\\\end{smallmatrix}\right).\end{align*}
We have $|A_{n}|=1$ and $A_{n}\neq I$. Then we assume the eigenvalues of $A_{n}$ are $\lambda^{2}_{n}, \frac{1}{\lambda^{2}_{n}}$, where $\lambda_{n}>1$.
That means
$\mbox{trace}A_{n}=|t_{n}|^{2}+\frac{1}{|t_{n}|^{2}}+|\frac{d_{n+1}}{t_{n}}-d_{n}t_{n}|^{2}=\lambda^{2}_{n}+\frac{1}{\lambda^{2}_{n}}.$
From the conditions, we have $\lambda^{2}_{n}+\frac{1}{\lambda^{2}_{n}}=\lambda^{2}_{g(n)}+\frac{1}{\lambda^{2}_{g(n)}}$, equivalently, $\lambda_{n}=\lambda_{g(n)}$.

Similar to the proof of Proposition \ref{220507.1}, we only need to verify that $A_{n}B_{n}\neq B_{n}A_{n}$ to prove that $T$ is an irreducible operator.

Let $A_{n}B_{n}=\left(\begin{smallmatrix}X^{(n)}_{11}&X^{(n)}_{12}\\X^{(n)}_{21}&X^{(n)}_{22}\\\end{smallmatrix}\right)$.
By a routine computation, we obtain
\begin{tiny}
\begin{equation*}
X^{(n)}_{11}=|t_{n}|^{2}|t_{n_1}|^{2}+|t_{n}|^{2}\left|\frac{d_{n}}{t_{n-1}}-d_{n-1}t_{n-1}\right|^{2}+\frac{\overline{ t_{n}}}{t_{n-1}}\left(\frac{d_{n+1}}{t_{n}}-d_{nt_{n}}\right)\left(\frac{\overline{ d_{n}}}{\overline{t_{n-1}}}-\overline{d_{n-1}}\overline{t_{n-1}}\right),
\end{equation*}
\end{tiny}
\begin{tiny}
\begin{equation*}
X^{(n)}_{12}=\frac{|t_{n}|^{2}}{\overline{ t_{n-1}}}\left(\frac{d_{n}}{t_{n-1}}-d_{n-1}t_{n-1}\right)+\frac{\overline{ t_{n}}}{|t_{n-1}|^{2}}\left(\frac{d_{n+1}}{t_{n}}-d_{n}t_{n}\right),
\end{equation*}
\end{tiny}
\begin{tiny}
\begin{equation*}
X^{(n)}_{21}=\left(\frac{\overline{d_{n+1}}t_{n}}{\overline{t_{n}}}-\overline{ d_{n}}|t_{n}|^{2}\right)\left(|t_{n-1}|^{2}+\left|\frac{d_{n}}{t_{n-1}}-d_{n-1}t_{n-1}\right|^{2}\right)
+\left(\left|\frac{d_{n+1}}{t_{n}}-d_{n}t_{n}\right|^{2}+\frac{1}{|t_{n}|^{2}}\right)\left(\frac{\overline{ d_{n}}}{|t_{n-1}|^{2}}-\frac{\overline{d_{n-1}}\overline{t_{n-1}}}{t_{n-1}}\right).
\end{equation*}
\end{tiny}

Then for the $(1,1)$-entry of the matrix $A_{n}B_{n}-B_{n}A_{n}=\left(\begin{smallmatrix}X^{(n)}_{11}-\overline{X^{(n)}_{11}}&X^{(n)}_{12}-\overline{ X^{(n)}_{21}}\\X^{(n)}_{21}-\overline{X^{(n)}_{12}}&X^{(n)}_{22}-\overline{ X^{(n)}_{22}}\\\end{smallmatrix}\right)$,
we obtain
\begin{tiny}
\begin{equation*}
X^{(n)}_{11}-\overline{X^{(n)}_{11}}=\frac{\overline{ t_{n}}}{t_{n-1}}(\frac{d_{n+1}}{t_{n}}-d_{n}t_{n})(\frac{\overline{d_{n}}}{\overline{ t_{n-1}}}-\overline{d_{n-1}}\overline{t_{n-1}})-\frac{t_{n}}{\overline{t_{n-1}}}(\frac{\overline{ d_{n+1}}}{\overline{t_{n}}}-\overline{d_{n}}\overline{ t_{n}})(\frac{d_{n}}{t_{n-1}}-d_{n-1}t_{n-1}).
\end{equation*}
\end{tiny}
Let $c_{n}=\frac{d_{n+1}}{t_{n}}-d_{n}t_{n}$, then $X^{(n)}_{11}-\overline{ X^{(n)}_{11}}=\frac{\overline{t_{n}}}{t_{n-1}}c_{n}\overline{c_{n-1}}-\frac{t_{n}}{\overline{ t_{n-1}}}\overline{c_{n}}c_{n-1}$.
Since $\overline{\frac{\overline{t_{n}}}{t_{n-1}}c_{n}\overline{c_{n-1}}}=\frac{t_{n}}{\overline{t_{n-1}}}\overline{c_{n}}c_{n-1}$, the difference between the two must be pure imaginary. It can be seen from the condition that $X^{(n)}_{11}-\overline{X^{(n)}_{11}}\neq0$. So it is not zero at the (1,1)-entry of $A_{n}B_{n}-B_{n}A_{n}$. Thus we conclude $A_{n}B_{n}-B_{n}A_{n}\neq0$.

Similarly, for the position $(2,1)$-entry of $A_{n}B_{n}-B_{n}A_{n}$, we have
\begin{tiny}
\begin{equation*}
X^{(n)}_{21}-\overline{X^{(n)}_{12}}=t_{n}\overline{ c_{n}}(|t_{n-1}|^{2}-\frac{1}{|t_{n-1}|^{2}}+|c_{n_1}|^{2})-\frac{1}{t_{n-1}}\overline{ c_{n-1}}(|t_{n}|^{2}-\frac{1}{|t_{n}|^{2}}-|c_{n}|^{2}), c_{n}=\frac{d_{n+1}}{t_{n}}-d_{n}t_{n}.
\end{equation*}
\end{tiny}
It can be seen from the condition that it is not zero at the (2,1)-entry of $A_{n}B_{n}-B_{n}A_{n}$. Thus we conclude $A_{n}B_{n}-B_{n}A_{n}\neq0$.
This completes the proof.
\end{proof}
In Theorem \ref{220507.2}, there are many choices for the bijection $g$ on $\mathbb{Z}$. We can define $g$ as $g(n)=-(n+i_0)$ for some $i_0\in\mathbb{Z}$ or the  identity mapping.
It is necessary to exists $N\in\mathbb{Z}$ such that when $n>N$, the sets of singular values of $A_n$ are distinct from each other.
In the following, we provide two examples.

\begin{ex}
Let $T=T(t_n,1)$, where $t_{n}=\frac{n+\frac{1}{2}+s}{n+\frac{1}{2}-s}$ and $s$ is a non zero pure imaginary number. Then $T$ is irreducible.
\end{ex}

\begin{proof}
Since $|t_{n}|^{2}=t_{n}\cdot\overline{ t_{n}}=\frac{n+\frac{1}{2}+s}{n+\frac{1}{2}-s}\cdot\frac{n+\frac{1}{2}-s}{n+\frac{1}{2}+s}=1$, we have
$$A_{n}=T^{*}_{n}T_{n}
=\left(\begin{smallmatrix}1&\frac{2(2n+1)s}{(n+\frac{1}{2}+s)^{2}}\\\frac{-2(2n+1)s}{(n+\frac{1}{2}-s)^{2}}&1+\big|\frac{2(2n+1)s}{(n+\frac{1}{2})^{2}-s^{2}}\big|^{2}\\\end{smallmatrix}\right),
B_{n}=T_{n-1}T^{*}_{n-1}
=\left(\begin{smallmatrix}1+\big|\frac{2(2n-1)s}{(n-\frac{1}{2})^{2}-s^{2}}\big|^{2}&\frac{2(2n-1)s}{(n-\frac{1}{2}-s)^{2}}\\\frac{-2(2n-1)s}{(n-\frac{1}{2}+s)^{2}}&1\\\end{smallmatrix}\right).
$$
Since $A_n\neq I$ and the determinant of $A_n$ is one for all $n\in\mathbb{Z}$. Let the eigenvalues of $A_n$ be $\lambda_n^2$ and $\frac{1}{\lambda_n^2}$ with $\lambda_n>1$.
Note that
$|\frac{2(2n+1)s}{(n+\frac{1}{2})^{2}-s^{2}}|^{2}=|\frac{2(2(-n-1)+1)s}{(-n-1+\frac{1}{2})^{2}-s^{2}}|^{2}$.
Then $\mbox{trace}A_{n}=\lambda^{2}_{n}+\frac{1}{\lambda^{2}_{n}}=\lambda^{2}_{-(n+1)}+\frac{1}{\lambda^{2}_{-(n+1)}}=\mbox{trace}A_{-(n+1)}$, therefore, $\lambda_{n}=\lambda_{-(n+1)}$.

Let $A_{n}B_{n}=\left(\begin{smallmatrix}X^{(n)}_{11}&X^{(n)}_{12}\\X^{(n)}_{21}&X^{(n)}_{22}\\\end{smallmatrix}\right)$.
By a routine computation, we obtain
$$X_{11}^{(n)}=1+\left|\frac{2(2n-1)s}{(n-\frac{1}{2})^{2}-s^{2}}\right|^{2}-\frac{4(2n+1)(2n-1)s^{2}}{(n+s+\frac{1}{2})^{2}(n+s-\frac{1}{2})^{2}}.$$
So for the matrix $A_{n}B_{n}-B_{n}A_{n}=\left(\begin{smallmatrix}X^{(n)}_{11}-\overline{X^{(n)}_{11}}&X^{(n)}_{12}-\overline{ X^{(n)}_{21}}\\X^{(n)}_{21}-\overline{X^{(n)}_{12}}&X^{(n)}_{22}-\overline{ X^{(n)}_{22}}\\\end{smallmatrix}\right),$
we have
$$X^{(n)}_{11}-\overline{X^{(n)}_{11}}=4(4n^{2}-1)s^{2}\left[\frac{1}{(n-s+\frac{1}{2})^{2}(n-s-\frac{1}{2})^{2}}-\frac{1}{(n+s+\frac{1}{2})^{2}(n+s-\frac{1}{2})^{2}}\right].$$
It is easy to see that $\frac{1}{(n-s+\frac{1}{2})^{2}(n-s-\frac{1}{2})^{2}}\notin\mathbb{R}$. So it is not zero at the (1,1)-entry of $A_{n}B_{n}-B_{n}A_{n}$. That is $A_{n}B_{n}-B_{n}A_{n}\neq0$. Similar to the proof of Proposition \ref{220507.1} and Theorem \ref{220507.2}, we deduce that $T$ is irreducible.
\end{proof}

\begin{ex}
Let $T=T(\sqrt{\frac{n+a}{n+b}},d_n)$, where $d_n>0$, $0<a<b<1$ and $a+b=1$.
If $d_{n}=d_{-n}$ and $\sqrt{\frac{n+a}{n+b}}c_{n}[(1+d^{2}_{n})\frac{n-a}{n-b}+(1+d^{2}_{n-1})\frac{n-b}{n-a}-2d_{n-1}d_{n}]\neq\sqrt{\frac{n-b}{n-a}}c_{n-1}[(1-d^{2}_{n})\frac{n+a}{n+b}-(1+d^{2}_{n+1})\frac{n+b}{n+a}+2d_{n}d_{n+1}]$ for $c_{n}=\sqrt{\frac{n+1-a}{n+1-b}}d_{n+1}-\sqrt{\frac{n+a}{n+b}}d_{n}$,
then $T$ is irreducible.
\end{ex}

\begin{proof}
Since $A_n\neq I$ and the determinant of $A_n$ is one for all $n\in\mathbb{Z}$. Let the eigenvalues of $A_n$ be $\lambda_n^2$ and $\frac{1}{\lambda_n^2}$ with $\lambda_n>1$.
We know that $\lambda_n=\lambda_m$ means that $\lambda_n^2+\frac{1}{\lambda_n^2}=\lambda_m^2+\frac{1}{\lambda_m^2}$.
Based on the condition,
we know that $d_{n}=d_{-n}$ is equivalent to $\lambda_{n}=\lambda_{-(n+1)}$, since $t^{2}_{n}=\frac{1}{t^{2}_{-(n+1)}}=\frac{n+a}{n+b}$.
By using the condition (2) of Theorem \ref{220507.2} and $c_{n}=\sqrt{\frac{n+1-a}{n+1-b}}d_{n+1}-\sqrt{\frac{n+a}{n+b}}d_{n}=\frac{d_{n+1}}{t_{n}}-d_{n}t_{n}$, we infer that $T$ is an irreducible operator.
\end{proof}

Because the scalar times identity operator is commutative with any operator, and in the case we are considering, some commutativity of the operator no longer naturally holds.
Therefore, we need additional conditions.
The objects of our study have a more complicated structure, which provides examples of other irreducible operators.

\subsection{The irreducibility of $T(w_n,v_n,d_n)$}

\begin{thm}\label{220508.6}
Let $T=T(w_n,v_n,d_n)$ with $|w_{n}|=|v_{n}|=1$. Suppose that there exists a bijective mapping $g$ on $\mathbb{Z}$ and $N\in\mathbb{Z}$ such that
the sets of singular values of $A_n$ are distinct from each other.
If $|d_{n+1}v_{n}-d_{n}w_{n}|^{2}=|d_{g(n)+1}v_{g(n)}-d_{g(n)}w_{g(n)}|^{2}$ and one of the following two conditions is true,
 \begin{enumerate}
 \item [(1)]
$d_{n+1}c_{n}+d_{n}c_{n-1}-d_{n+1}\overline{w_{n}}v_{n}c_{n-1}\notin\mathbb{R}$, where $c_{n}=\overline{d_{n}}\overline{w_{n}}v_{n}$;
 \item [(2)]
$(\overline{d_{n+1}}w_{n}v_{n}-\overline{d_{n}})|d_{n}v_{n-1}-d_{n-1}w_{n-1}|^{2}\neq(\overline{d_{n-1}}\overline{w_{n-1}}v_{n-1}-\overline{d_{n}})|d_{n+1}v_{n}-d_{n}w_{n}|^{2}$,
\end{enumerate}
then T is irreducible.
\end{thm}

\begin{proof}
Since $T_{n}=\left (\begin{smallmatrix}w_{n}&d_{n+1}v_{n}-d_{n}w_{n}\\0&v_{n}\\\end{smallmatrix}\right )$ and
\begin{align*}A_{n}&
=\left (\begin{smallmatrix}1&\overline{w_{n}}d_{n+1}v_{n}-d_{n}\\\overline{d_{n+1}}\overline{v_{n}}w_{n}-\overline{d_{n}}&1+|d_{n+1}v_{n}-d_{n}w_{n}|^{2}\\\end{smallmatrix}\right ),\\
B_{n}&
=\left (\begin{smallmatrix}1+|d_{n}v_{n-1}-d_{n-1}w_{n-1}|^{2}&d_{n}-d_{n-1}w_{n-1}\overline{v_{n-1}}\\\overline{d_{n}}-\overline{d_{n-1}}\overline{w_{n-1}}v_{n-1}&1\\\end{smallmatrix}\right ).\end{align*}
A simple calculation can be obtained the determinant of $A_{n}$ as 1 and $A_{n}\neq I$. Then assume the eigenvalues of $A_{n}$ are $\lambda^{2}_{n}, \frac{1}{\lambda^{2}_{n}}$, where $\lambda_{n}>1$.
That means $\mbox{trace}A_{n}=2+|d_{n+1}v_{n}-d_{n}w_{n}|^{2}=\lambda^{2}_{n}+\frac{1}{\lambda^{2}_{n}}$.
If $|d_{n+1}v_{n}-d_{n}w_{n}|^{2}=|d_{g(n)+1}v_{g(n)}-d_{g_(n)}w_{g(n)}|^{2}$, this is equivalent to
$\lambda^{2}_{n}+\frac{1}{\lambda^{2}_{n}}=\lambda^{2}_{g(n)}+\frac{1}{\lambda^{2}_{g(n)}}$.
Thus, $\lambda_{n}=\lambda_{g(n)}$.

By a routine calculation, we have the $(1,1)$-entry and $(2,1)$-entry of $A_{n}B_{n}-B_{n}A_{n}$ are
$d_{n+1}\overline{d_{n}}\overline{w_{n}}v_{n}+d_{n}\overline{d_{n-1}}\overline{w_{n-1}}v_{n-1}-d_{n+1}\overline{d_{n-1}}\overline{w_{n}}\overline{w_{n-1}}v_{n}v_{n-1}-\overline{d_{n+1}}d_{n}w_{n}\overline{v_{n}}-\overline{d_{n}}d_{n-1}w_{n-1}\overline{v_{n-1}}+\overline{d_{n+1}}d_{n-1}w_{n}w_{n-1}\overline{v_{n}}\overline{v_{n-1}}$
and $(\overline{d_{n+1}}w_{n}v_{n}-\overline{d_{n}})|d_{n}v_{n-1}-d_{n-1}w_{n-1}|^{2}-(\overline{d_{n-1}}\overline{w_{n-1}}v_{n-1}-\overline{d_{n}})|d_{n+1}v_{n}-d_{n}w_{n}|^{2}$.
It can be seen that the first three terms and the last three terms are conjugate to each other at the (1,1)-entry. From condition (1), we know that $d_{n+1}c_{n}+d_{n}c_{n-1}-d_{n+1}\overline{w_{n}}v_{n}c_{n-1}\notin\mathbb{R}$, that is to say, the $(1,1)$-entry of  $A_{n}B_{n}-B_{n}A_{n}$ is not zero. Thus we conclude that $A_{n}B_{n}-B_{n}A_{n}\neq0$.
In the same way, from condition (2), we see that $(2,1)$-entry isn't zero.
Similar to the proof of Proposition \ref{220507.1}, we deduce that $T$ is irreducible.
\end{proof}

\begin{cor}
Let $T=T(w_n,v_n,\alpha)$. If $|w_{n}-v_{n}|^{2}=|w_{m}-v_{m}|^{2}$ only holds when $m=-(n+i_0)$ for some $i_0\in\mathbb{Z}$ and one of the following two conditions is true,
\begin{enumerate}
 \item [(1)]
$w_{n}w_{n-1}\overline{c_{n}}\overline{c_{n-1}}\notin\mathbb{R}$, where $c_{n}=w_{n}-v_{n}$;
 \item [(2)]
$(1-w_{n}\overline{v_{n}})|w_{n-1}-v_{n-1}|^{2}\neq(1-\overline{ w_{n-1}}v_{n-1})|w_{n}-v_{n}|^{2}$,
\end{enumerate}
then $T$ is irreducible.
\end{cor}

In the following, we will use another method to provide a sufficient condition for $T(w_n,v_n,d_n)$ to be an irreducible operator.
Before that, let us introduce an equivalence theorem on shift operators proved by A.L. Shields in \cite{m}.
Although this result is well-known, we record it for overall completeness and later use.

\begin{thm}\label{157}\cite{m}
Let $S$ and $T$ be bilateral weighted shifts with weight sequences $\{v_{n}\}_{n\in\mathbb{Z}}$, $\{w_{n}\}_{n\in\mathbb{Z}}$ and $w_n,v_n\neq0$ for all $n\in\mathbb{Z}$. Then $S\sim_sT$ if and only if there exist $k\in\mathbb{Z}$ and $c_{1},c_{2}$ such that $0<c_{1}\leq|\frac{w_{k+m}\cdots w_{k+n}}{v_{m}\cdots v_{n}}|\leq c_{2}$ for all $m\leq n$.
Specially, if $S$ and $T$ be unilateral weighted shifts with no zero weights, then $S\sim_sT$ if and only if the formula above holds with $k=m=0$ for all $n\geq0$.
\end{thm}

If we consider the case where $S\sim_uT$, then the two equivalent conditions of the previous theorem are replaced by $|v_n|=|w_{n+k}|$ and
$|v_n|=|w_{n}|$.
This means that the shift operator with weight sequence $\{w_n\}_{n\in I}$ and the shift operator with weight sequence $\{|w_n|\}_{n\in I}$ are always unitarily equivalent, $I=\mathbb{Z}^+$ or $\mathbb{Z}$.
Recall that by a theorem of R.L. Kelly and N.K. Nikolskii in \cite{Halmos,NKN}, which gives a characterization that a bilateral weighted shift is reducible.

\begin{lem}\cite{Halmos,NKN}\label{220508.1}
If $T$ is a bilateral weighted shift with strictly positive weights $w_{n}$, $n\in\mathbb{Z}$, then a necessary and sufficient condition that $T$ be reducible is that the sequence $\{w_{n}\}_{n\in\mathbb{Z}}$ be periodic.
\end{lem}

\begin{prop}\label{220508.4}
Let $T=T(w_n,v_n,d_n)$. If both $T_{0}$ and $T_{1}$ are irreducible and $\lim\limits_{k\to\infty}\prod\limits^{k}_{m=0}\frac{|w_{j+m}|}{|v_{i+m}|}=0$ for any $i,j\in\mathbb{Z}$, then $T$ is irreducible.
\end{prop}

\begin{proof}
To prove that $T$ is irreducible, we just need to verify that for any $P\in\mathcal{A}'(T)$ and $P=P^{*}=P^{2}$, if $T$ satisfies the condition in Theorem, then $P=0$ or $I$.

Let $P\in\mathcal{A}'(T)$ and $P$ under decomposition $l^{2}(\mathbb{Z})\oplus l^{2}(\mathbb{Z})$ be $\left (\begin{smallmatrix}P_{00}&P_{01}\\P_{10}&P_{11}\\\end{smallmatrix}\right)$.
Then $P_{00}=P^{*}_{00}, P_{01}=P^{*}_{10},P_{11}=P^{*}_{11}$.
Since $P\in\mathcal{A}'(T)$, we have
\begin{equation}\label{2023.10.231}
\left(\begin{smallmatrix}P_{00}T_{0}&P_{00}(XT_{1}-T_{0}X)+P_{01}T_{1}\\P_{10}T_{0}&P_{10}(XT_{1}-T_{0}X)+P_{11}T_{1}\\\end{smallmatrix}\right)
=\left(\begin{smallmatrix}T_{0}P_{00}+(XT_{1}-T_{0}X)P_{10}&T_{0}P_{01}+(XT_{1}-T_{0}X)P_{11}\\T_{1}P_{10}&T_{1}P_{11}\\\end{smallmatrix}\right).
\end{equation}
For any $i,j\in\mathbb{Z}$ , let $a_{ij}=\langle P_{00}e_{j},e_{i}\rangle,b_{ij}=\langle P_{01}e_{j},e_{i}\rangle,c_{ij}=\langle P_{10}e_{j},e_{i}\rangle,d_{ij}=\langle P_{11}e_{j},e_{i}\rangle$. Then the matrix representations of $P_{00},P_{01},P_{10},P_{11}$ under $\{e_{n}\}_{n\in\mathbb{Z}}$ are $P_{00}=(a_{ij})_{i,j\in\mathbb{Z}},P_{01}=(b_{ij})_{i,j\in\mathbb{Z}},P_{10}=(c_{ij})_{i,j\in\mathbb{Z}},P_{11}=(d_{ij})_{i,j\in\mathbb{Z}}$. It follows that $a_{ij}=\overline{a_{ji}}$, $d_{ij}=\overline{d_{ji}}$ and $b_{ij}=\overline{c_{ji}}$.

For the (2,1)-entry of (\ref{2023.10.231}), we have $P_{10}T_{0}=T_{1}P_{10}$, that is,  $w_{j}c_{i+1,j+1}=v_{i}c_{i,j}\,\,\text{for\,\,}i,j\in\mathbb{Z}$.
Now suppose $c_{i,j}\neq0$, because of the non-zero property of the weighted sequences $\{w_{i}\}_{i\in\mathbb{Z}}, \{v_{i}\}_{i\in\mathbb{Z}}$, we have $c_{i+1,j+1}\neq0$. Replacing $i, j$ by $i+1, j+1$, we have $c_{i+2,j+2}\neq0$.
Repeating this process, we get $c_{i+k,j+k}\neq0,k\in\mathbb{Z}^+$.

We know $\|P\|=1$, then $\|P_{ij}\|\leq1$ for $i,j=0,1$ and
$|c_{ij}|\leq\|P_{10}\|\leq1. $
Note that the equation $|c_{ij}|^{2}=(\prod\limits_{m=0}^{k}\frac{|w_{j+m}|^{2}}{|v_{i+m}|^{2}})|c_{i+k+1,j+k+1}|^{2}$ always holds for any $k\in\mathbb{Z}^{+}$. Since $\lim\limits_{k\to\infty}\prod\limits^{k}_{m=0}\frac{|w_{j+m}|^{2}}{|v_{i+m}|^{2}}=0$ and $|c_{i+k+1,j+k+1}|^{2}\leq1$, $c_{i,j}=0$, which is contradictory to the assumption $c_{i,j}\neq0$. So the assumption is not valid and then $c_{i,j}=0$ for any $i,j\in\mathbb{Z}$.
So $P_{10}=0,P_{01}=0$.

For the (1,1)-entry of (\ref{2023.10.231}), we have $P_{00}T_{0}=T_{0}P_{00}$, that is,
$w_{j}a_{i+1,j+1}=w_{i}a_{i,j},i,j\in\mathbb{Z}$.
Letting $i=j$, we have $a_{i,i}=a_{i+1,i+1}$, since the weight sequence is not zero. From the arbitrariness of $i$, we can know that the elements on the diagonal of matrix $P_{00}$ are always equal. Suppose $a_{i,j}\neq0$, then $a_{i+k,j+k}\neq0$ for $k\in\mathbb{Z}^{+}$ and
$\frac{w_{j}}{w_{i}}=\frac{a_{i,j}}{a_{i+1,j+1}}$.
By taking conjugate, we obtain $T^{*}_{0}P_{00}=P_{00}T^{*}_{0}$ and $\frac{\overline{w_{i}}}{\overline{ w_{j}}}=\frac{a_{i,j}}{a_{i+1,j+1}}$.
This implies that $|w_{i}|^{2}=|w_{j}|^{2}$.
Similarly, we conclude $|w_{i+k}|^{2}=|w_{j+k}|^{2}$ for any $k\in\mathbb{Z}^{+}$. When $i\neq j$, there exists $s\neq0,s\in\mathbb{Z}$ such that $i+s=j$. Then $|w_{i}|^{2}=|w_{i+s}|^{2}=|w_{i+2s}|^{2}=\ldots=|w_{i+ns}|^{2}$ for $n\in\mathbb{Z}^{+}$.
We can see that the modulus of the sequence changes periodically.
Let $\tilde T_{0}$ be a weighted shift concerning the orthonormal basis $\{e_n\}_{n\in\mathbb{Z}}$ and $\tilde T_{0}e_{n}=|w_{n}|e_{n+1}$.
So $\tilde T_{0}$ is a reducible operator according to Lemma \ref{220508.1}.
Since $T_{0}\sim_{u}\tilde T_{0}$, $T_{0}$ is also a reducible operator, which is contradictory to the condition. So the assumption is not valid, then $a_{i,j}=0$ for $i\neq j$.
Then $P_{00}$ is a diagonal operator whose diagonal elements are always equal.
Because of the self-adjoint of $P_{00}$, there exists $l_{1}\in\mathbb{R}$ such that $P_{00}=l_{1}I$, where $I$ is the identity of $l^{2}(\mathbb{Z})$. Similarly, for the (2,2)-entry of (\ref{2023.10.231}), we can also get $P_{11}=l_{2}I$, where $l_{2}\in\mathbb{R}$.

For the (1,2)-entry of (\ref{2023.10.231}), $P_{00}(XT_{1}-T_{0}X)=(XT_{1}-T_{0}X)P_{11}$ and $XT_{1}-T_{0}X\neq0$, we have $l_{1}=l_{2}$.
Then $P=l_{1}I.$
Since $P=P^{2}$, then $P=0$ or $I$, which implies that $T$ is irreducible.
\end{proof}

The following corollary is an immediate consequence of Proposition \ref{220508.4}.
\begin{cor}
Let $T=T(t_n,d_n)$. If both $T_{0}$ and $T_{1}$ are irreducible and $\lim\limits_{k\to\infty}\prod\limits^{k}_{m=0}|t_{i+m}|=0$ for any $i\in\mathbb{Z}$, then $T$ is irreducible.
\end{cor}

\section{Weakly homogeneity of a class of operators}

In \cite{e,d}, it has been proven that if $T_{0}$ and $T_{1}$ are homogeneous operators and have the same unitary representation $U(g)$, then for any $\alpha\in\mathbb{C}$, the operator $\widetilde{T}=\left(\begin{smallmatrix}T_{0}&\alpha(T_{0}-T_{1})\\0&T_{1}\\\end{smallmatrix}\right)$ is homogeneous with the unitary representation $U(g)\oplus U(g)$.
In other words, the homogeneity of $T_0$ and $T_1$ can be completely transferred to $\widetilde{T}$ under its construction.
Here we give two examples to show that when the $2\times2$ operator matrix does not satisfy this construction, it no longer remains homogeneous but is weakly homogeneous.
Let $T=\left(\begin{smallmatrix}T_{0}&XT_{1}-T_{0}X\\0&T_{1}\\\end{smallmatrix}\right)$, where $T_{0}$ and $T_{1}$ are homogeneous operators and $X$ is a bounded linear operator.
Let $S=\left(\begin{smallmatrix}I&-X\\0&I\\\end{smallmatrix}\right)$. It is easy to see that $S$ is invertible and its inverse is $\left(\begin{smallmatrix}I&X\\0&I\\\end{smallmatrix}\right)$.
Thus, $T=S^{-1}(T_{0}\oplus T_{1})S$.
Assume that the unitary representation of $T_0$ and $T_1$ are $U(g)$ and $V(g)$ respectively, for any $g\in M\ddot{o}b$.
Then we have $g(T_{0})=U(g)^{-1}T_{0}U(g)$, $g(T_{1})=V(g)^{-1}T_{1}V(g)$ and
\begin{equation}\label{20230303}
g(T_{0}\oplus T_{1})=g(T_{0})\oplus g(T_{1})=(U(g)\oplus V(g))^{-1}(T_{0}\oplus T_{1})(U(g)\oplus V(g)).
\end{equation}
For a fixed but arbitrary $g\in M\ddot{o}b$, it follows from $g$ is analytic and equation (\ref{20230303}) that
\begin{equation}\label{202303031}
\begin{array}{lll}
g(T)&=&g(S^{-1}(T_{0}\oplus T_{1})S)\\
&=&S^{-1}g(T_{0}\oplus T_{1})S\\
&=&S^{-1}(U(g)\oplus V(g))^{-1}(T_{0}\oplus T_{1})(U(g)\oplus V(g))S\\
&=&S^{-1}(U(g)\oplus V(g))^{-1}STS^{-1}(U(g)\oplus V(g))S.
\end{array}
\end{equation}
Since $S^{-1}(U(g)\oplus V(g))^{-1}S$ must be invertible, $T$ must be a weakly homogeneous operator.
If you want to get that $T$ is homogeneous, you need $S^{-1}(U(g)\oplus V(g))^{-1}S$ to be a unitary operator.
Note that
\begin{equation}\label{202303032}
S^{-1}(U(g)\oplus V(g))^{-1}S=\left(\begin{smallmatrix}U(g)^{-1}&XV(g)^{-1}-U(g)^{-1}X\\0&V(g)^{-1}\\\end{smallmatrix}\right).
\end{equation}
A routine verification shows that $S^{-1}(U(g)\oplus V(g))^{-1}S$ is a unitary operator if and only if $XV(g)^{-1}=U(g)^{-1}X$.
Next, we will explain that even in the following two special cases, $T$ cannot be identified with a homogeneous.
Let $T_{0}$ and $T_{1}$ have the same unitary representation $U(g)$ and $X\neq \alpha I$ for any $g\in M\ddot{o}b$ and $\alpha\in \mathbb{C}$.
Since $X\in\mathcal{L}(\mathcal{H})$ is a bounded linear operator arbitrarily, this means that apart from the scalar operator $\alpha I$, we cannot be sure that the operator $X$ in other operator classes satisfies $XU(g)^{-1}=U(g)^{-1}X$.
When $T_{0}$ and $T_{1}$ have different unitary representations $U(g)$ and $V(g)$ and $X=\alpha I$ for some non-zero $\alpha\in\mathbb{C}$,
then $\alpha V(g)^{-1}=\alpha U(g)^{-1}$ is equivalent to $V(g)=U(g)$ for any $g\in M\ddot{o}b$.

Therefore, in the following we consider the weak homogeneity problem associated with the operator
$T=\left(\begin{smallmatrix}T_{0}&XT_{1}-T_{0}X\\0&T_{1}\\\end{smallmatrix}\right)$.
We record the next fact as a Lemma for our later use.

\begin{lem}\label{220509.9}
Let $T=\left(\begin{smallmatrix}T_{0}&XT_{1}-T_{0}X\\0&T_{1}\\\end{smallmatrix}\right)$ for a bounded linear operator $X$.
If $T_{0}$, $T_{1}$ are weakly homogeneous operators, then so is $T$.
\end{lem}

\subsection{Strong irreducible decompositions and weak homogeneity}
Let $T_{i}\in\mathcal{L}(\mathcal{H}_i)$, $i=0,1$.
Define $\sigma_{T_{0},T_{1}}$ be the operator $\sigma_{T_{0},T_{1}}(X)=T_{0}X-XT_{1}$ for $X\in\mathcal{L}(\mathcal{H}_{1},\mathcal{H}_{0})$.
Let $\sigma_{T_{0}}: \mathcal{L}(\mathcal{H}_{0})\to\mathcal{L}(\mathcal{H}_{0})$ be the operator $\sigma_{T_{0},T_{0}}$. Then $X\in\mbox{ker}\sigma_{T_{0},T_{1}}$ means that $T_{0}X=XT_{1}$.
Recall Definition 4.1 introduced in \cite{n}.
We say that $T=\left(\begin{smallmatrix}T_{0}&XT_{1}-T_{0}X\\0&T_{0}\\\end{smallmatrix}\right)$ belongs to the set $N\mathcal{F}\mathcal{B}_{2n}(\Omega)$, if $T_{0},T_{1}\in\mathcal{B}_{n}(\Omega)$ and $X\notin\mbox{ker}\sigma_{T_{0},T_{1}}$.
It is natural to ask: when does the converse of Lemma \ref{220509.9} hold?
There is an answer in \cite{n}.
If $\left(\begin{smallmatrix}T_{0}&XT_{1}-T_{0}X\\0&T_{0}\\\end{smallmatrix}\right)$ is weakly homogeneous for some bounded linear operator $X$ and for any $\phi_{\alpha}\in M\ddot{o}b$, $\mbox{ker}\sigma_{\phi_{\alpha}(T_{0}),T_{1}}=\{0\}$, then $T_{0},\,\,T_{1}$ are both weakly homogeneous.
By using the $K_0$-group of the commutant algebra, we obtain the following results.

\begin{thm}\label{202312.233}
Let $T=\left(\begin{smallmatrix}T_{0}&XT_{0}-T_{0}X\\0&T_{0}\\\end{smallmatrix}\right)\in N\mathcal{F}\mathcal{B}_{2n}(\mathbb{D})$, where $T_{0}\in(SI)$. Then $T$ is weakly homogeneous if and only if $T_{0}$ is weakly homogeneous.
\end{thm}

\begin{proof}
By Lemma \ref{220509.9}, the sufficiency is valid.
Let $S=\left(\begin{smallmatrix}I&-X\\0&I\\\end{smallmatrix}\right)$. Then $T=S^{-1}(T_{0}\oplus T_{0})S$.
We know that similar transformations do not change the spectrum, then $\sigma(T)=\sigma(T_{0}\oplus T_{0})=\sigma(T_{0})=\overline{\mathbb{D}}$,
since the spectrum of a weakly homogeneous operator is $\mathbb{T}$ or $\overline{\mathbb{D}}$ by using the spectral mapping theorem (see \cite{020}).

Since $T$ is weakly homogeneous, for any $g\in M\ddot{o}b$, there exists an invertible operator $\tilde{S_{g}}$ such that $g(T)=\tilde{S_{g}}^{-1}T\tilde{S_{g}}$. Note that
$g(T)=S^{-1}(g(T_{0}\oplus T_{0}))S=\tilde{S_{g}}^{-1}S^{-1}(T_{0}\oplus T_{0})S\tilde{S_{g}}$,
then $T_{0}\oplus T_{0}$ is weakly homogeneous since for any $g\in M\ddot{o}b$, $S\tilde{S_{g}}^{-1}S^{-1}$ is invertible and
$g(T_{0})^{(2)}=(S\tilde{S_{g}}^{-1}S^{-1})T_{0}^{(2)}(S\tilde{S_{g}}S^{-1})$.
It implies that $(S\tilde{S_{g}}^{-1}S^{-1}\oplus I)(g(T_{0})^{(2)}\oplus T_{0}^{(2)})(S\tilde{S_{g}}S^{-1}\oplus I)=T_{0}^{(4)}$ for all $g\in M\ddot{o}b$.
By Theorem \ref{220506.1} and $T_{0}\in(SI)$, we obtain
\begin{equation}\label{220509.2}
K_{0}(\mathcal{A}'(g(T_{0})^{(2)}\oplus T^{(2)}_{0}))\cong K_{0}(\mathcal{A}'(T^{(4)}_{0})) \cong \mathbb{Z}.
\end{equation}

We claim that $g(T_{0})$ is also strongly irreducible. Assume that $g(T_{0})$ is not strongly irreducible.
Since every Cowen-Douglas operator can be written as the direct sum of finitely many strongly irreducible Cowen-Douglas operators in \cite{o}.

\textbf{Case 1}: Without loss of generality, we suppose that $g(T_{0})$ has a finite (SI) decomposition $g(T_{0})=T^{(n_{1})}_{1}\oplus T^{(n_{2})}_{2}\oplus\cdots\oplus T^{(n_{l})}_{l}$ for some positive integer $l>1$.
Then $g(T_{0})^{(2)}$ has a finite (SI) decomposition $g(T_{0})^{(2)}=T^{(2n_{1})}_{1}\oplus T^{(2n_{2})}_{2}\oplus\cdots\oplus T^{(2n_{l})}_{l}$, $l>1$. By Theorem \ref{220506.1}, we have $K_{0}(\mathcal{A}'(g(T_{0})^{(2)}\oplus T^{(2)}_{0}))\cong \mathbb{Z}^{l}\,\,(l>1)$, which is contradictory to (\ref{220509.2}). Therefore, the assumption is contradictory.

\textbf{Case 2}:  Let $g(T_{0})$ have a finite (SI) decomposition $g(T_{0})=T_{1}^{(n_{1})},n_{1}>1$, where $T_{1}\in\mathcal{B}_{\frac{n}{n_{1}}}(\mathbb{D})$.
Then $g(T_{0})^{(2)}=T_{1}^{(2n_{1})}$ is a finite (SI) decomposition.
By using Theorem \ref{220506.1} again, we obtain $K_{0}(\mathcal{A}'(g(T_{0})^{(2)}\oplus T_{0}^{(2)}))=K_{0}(\mathcal{A}'(T_{1}^{(2n_{1})}\oplus T_{0}^{(2)}))\cong\mathbb{Z}^{2}$, which is contradictory to (\ref{220509.2}).

So $g(T_{0})$ is strongly irreducible. It follows from $T_{0},g(T_{0})\in(SI),\,T^{(2)}_{0}\sim_{s}g(T_{0})^{(2)}$ and Theorem \ref{220506.2} that $T_{0}\sim_{s}g(T_{0})$ for any $g\in M\ddot{o}b$. Hence, $T_{0}$ is weakly homogeneous.
\end{proof}
When $n=1$, C.K. Fong and C.L. Jiang have shown that $\mathcal{B}_{1}(\Omega)\subset(SI)$ in \cite{018}.
Therefore, when $T$ belongs to $N\mathcal{F}\mathcal{B}_{2}(\mathbb{D})$ in Theorem \ref{202312.233}, $T_{0}\in(SI)$ is naturally satisfied.
From Theorem \ref{202312.233}, we have directly the following corollary, which does not restrict the operator $T$ to be a Cowen-Douglas operator.
\begin{cor}
Let $T=\left(\begin{smallmatrix}T_{0}&XT_{0}-T_{0}X\\0&T_{0}\\\end{smallmatrix}\right)$, where $T_{0}\in(SI)$ and $X\notin\mbox{ker}\sigma_{T_{0}}$. Suppose that for any $g\in M\ddot{o}b$, $g(T_{0})\in(SI)$. Then $T$ is weakly homogeneous if and only if $T_{0}$ is weakly homogeneous.
\end{cor}

\begin{thm}
Let $T=\left(\begin{smallmatrix}T_{0}&XT_{1}-T_{0}X\\0&T_{1}\\\end{smallmatrix}\right)\in N\mathcal{FB}_{2n}(\mathbb{D})$, where $T_{i}\in(SI)$, $i=1,2$. Suppose $T_{0}\nsim_{s}T_{1}$($T_{0}$ is not similar to $T_{1}$), if $T$ and $T_{0}$ are weakly homogeneous, then so is $T_{1}$.
\end{thm}

\begin{proof}
Since $T$ is weakly homogeneous, by using Theorem \ref{202312.233}, we have $T_{0}\oplus T_{1}$ is weakly homogeneous, that is, $g(T_{0})\oplus g(T_{1})\sim_{s}T_{0}\oplus T_{1}$ for any $g\in M\ddot{o}b$. From $T_{0},T_1\in(SI)$ and $T_{0}\nsim_{s}T_{1}$, $T_{0}\oplus T_{1}$ is a (SI) decomposition. By Theorem \ref{220506.1}, we have
\begin{equation}\label{220509.6}
\begin{array}{lll}
K_{0}(\mathcal{A}'(T_{0}\oplus T_{1}\oplus g(T_{0})\oplus g(T_{1})))
&\cong& K_{0}(\mathcal{A}'(g(T_{0})\oplus g(T_{1})))\\
&\cong& K_{0}(\mathcal{A}'(T_{0}\oplus T_{1}))\\
&\cong& \mathbb{Z}^{2}.
\end{array}
\end{equation}
We claim that $g(T_{i})$ is also strongly irreducible for $i=0,1$. Otherwise, assume that $g(T_{0})$ is not strongly irreducible. Since every Cowen-Douglas operator can be written as the direct sum of finitely many strongly irreducible Cowen-Douglas operators \cite{o}.
Suppose that $g(T_{0})$ has a finite (SI) decomposition $g(T_{0})=T^{(n_{1})}_{01}\oplus T^{(n_{2})}_{02}\oplus\cdots\oplus T^{(n_{k})}_{0k}$, $k\geq1$. Then by Theorem \ref{220506.1} again, we have
\begin{equation*}
\begin{array}{lll}
K_{0}(\mathcal{A}'(T_{0}\oplus T_{1}\oplus g(T_{0})\oplus g(T_{1})))
&\cong& K_{0}(\mathcal{A}'(T_{0}\oplus T_{1}\oplus T^{(n_{1})}_{01}\oplus\cdots\oplus T^{(n_{k})}_{0k}\oplus g(T_{1})))\\
&\cong& \mathbb{Z}^{l},
\end{array}
\end{equation*}
where $l\geq k+2\geq3$, which contradicts (\ref{220509.6}). Therefore, the assumption is invalid and $g(T_{0})$ is strongly irreducible.
By the same process, we can also obtain that $g(T_{1})$ is strongly irreducible.

By using (\ref{220509.6}) again, we have $g(T_{0})\nsim_{s}g(T_{1})$. Since $T_{0}$ is weakly homogeneous, that is, $T_{0}\sim_{s}g(T_{0})$ for any $g\in M\ddot{o}b$, we obtain $g(T_{1})\sim_{s}T_{1}$ for any $g\in M\ddot{o}b$, that is, $T_{1}$ is weakly homogeneous.
\end{proof}

\subsection{Reproducing kernels and weak homogeneity}
Some basic facts and notations related to weak homogeneity are introduced in Subsection \ref{2.3who}.
Recall that an operator $T$ on a Hilbert space is said to be M$\ddot{o}$bius bounded if the family $\{\phi(T):\phi\in M\ddot{o}b\}$
is uniformly bounded in norm (see \cite{ALS}).
An operator similar to a homogeneous operator must be a M$\ddot{o}$bius bounded weakly homogeneous operator, but the converse is false, such as $M_z$ on Dirichlet space and $T_{\lambda,s}$, they are weakly homogeneous, but not M$\ddot{o}$bius bounded, thus not similar to any homogeneous operator due to B. Bagchi and G. Misra in \cite{d}, where $T_{\lambda,s}$ is a shift with weight
$\{\frac{n+(1+\lambda)\backslash2+s}{n+(1+\lambda)\backslash2-s}\}_{n\in\mathbb{Z}}$,
$-1<\lambda\leq1$, $Re(s)>\frac{1}{2}$ and $Im(s)>0$.
In what follows, we give some discussion of these concepts. Let us begin with a fact.

\begin{lem}\label{2310.241}
Let $K_0:\Omega\times\Omega\rightarrow \mathbb{C}$ be a positive kernel and $K_1:\Omega\times\Omega\rightarrow \mathcal{M}_k(\mathbb{C})$ for $k\geq1$ and $K_1(z,z)$ is invertible.
If $(M_\phi,\mathcal{H}_{K_0})$ is bounded for some holomorphic function $\phi:\Omega\rightarrow\mathbb{C}$, then so is $(M_\phi,\mathcal{H}_{K_0K_1})$.
\end{lem}
\begin{proof}
By Lemma 2.7 of \cite{SGGM}, we know that if $(M_\phi,\mathcal{H}_{K_0})$ is bounded, then there exists $c>0$ such that $(c^2-\phi(z)\overline{\phi(w)})K_0(z,w)$ is a nonnegative kernel on $\Omega\times\Omega$.
The minimum of all $c$ satisfying this property is its norm.
So we can choose the appropriate $c$ so that $(c^2-\phi(z)\overline{\phi(w)})K_0(z,w)$ is a positive definite kernel.
Then $K_0(z,w)K_1(z,w)$ and $(c^2-\phi(z)\overline{\phi(w)})K_0(z,w)K_1(z,w)$ are positive kernels. Thus, $(M_\phi,\mathcal{H}_{K_0K_1})$ is bounded.
\end{proof}

Let $K_{0},K_{1}:\Omega\times\Omega\to\mathbb{C}$ be two non-negative definite kernels and
$J_{k}(K_{0},K_{1})|_{res\Delta}(z,w)=(K_{0}(z,w)\partial^{i}\overline{\partial}^{j}K_{1}(z,w))^{k}_{i,j=0}$, $z,w\in\Omega$ for non negative integer $k$, where $res\Delta$ means restricted to the diagonal set.
Clearly, it is a non negative definite kernel.
If $K_{0},K_{1}$ are two sharp kernels, then so is the matrix-valued kernel function $J_{k}(K_{0},K_{1})|_{res\Delta}$ due to S. Ghara in \cite{i}.
The following result is also given:

\begin{lem}\cite{i}\label{220509.7}
Suppose that $K_{0}$ and $K_{1}$ are two scalar valued sharp positive definite kernels on $\mathbb{D}\times\mathbb{D}$. If the multiplication operators $M_{z}$ on $(\mathcal{H}, K_{0})$ and $(\mathcal{H},K_{1})$ are weakly homogenous, then $M_{z}$ on $(\mathcal{H},J_{k}(K_{0},K_{1})|_{res\Delta})$ is also weakly homogenous.
\end{lem}
The next result is a generalization of the result given by S. Ghara in \cite{i} to operator class $N\mathcal{FB}_2(\mathbb{D})$.

\begin{prop}\label{220509.10}
Let $(M^{*}_{z},\mathcal{H}_{K})\in\mathcal{B}_{1}(\mathbb{D})$ be a weakly homogeneous operator.
Let $T=\left(\begin{smallmatrix}T_{0}&XT_{1}-T_{0}X\\0&T_{1}\\\end{smallmatrix}\right)\in N\mathcal{FB}_2(\mathbb{D})$, where $T_{i}\sim_{u}(M^{*}_{z},\mathcal{H}_{K_{i}})$ and $K_{i}=KK^{(\lambda_{i})}$, $\lambda_{i}>0$, $i=0,1$.
If for any $\lambda>0$,
the limits of $\lim\limits_{w\rightarrow\partial\mathbb{D}}K(\bar{w},\bar{w})K^{(\lambda)}(\bar{w},\bar{w})$ and $\lim\limits_{w\rightarrow\partial\mathbb{D}}\frac{K(\bar{w},\bar{w})}{K^{(\lambda)}(\bar{w},\bar{w})}$ are either 0 or $\infty$,
then $T$ is a $M\ddot{o}bius$ bounded weakly homogeneous operator which is not similar to the direct sum of any homogeneous operators.
\end{prop}

\begin{proof}
Since $(M_{z},\mathcal{H}_{K^{(\lambda_{i})}})$ is homogeneous proven by G. Misra in \cite{la}, it follows from Lemma \ref{220509.7} that $(M_{z},\mathcal{H}_{KK^{(\lambda_{i})}})$ is weakly homogeneous, so are $T_{i}$ and $T$ by Lemma \ref{220509.9}, $i=0,1$.

Let $S=\left(\begin{smallmatrix}I&-X\\0&I\\\end{smallmatrix}\right)$. We have $T=S^{-1}(T_{0}\oplus T_{1})S$. Then for any $g\in M\ddot{o}b$, we have $g(T)=S^{-1}(g(T_{0}\oplus T_{1}))S=S^{-1}(g(T_{0})\oplus g(T_{1}))S$. So we obtain
$$\|g(T)\|\leq\|S^{-1}\|\|g(T_{0})\oplus g(T_{1})\|\|S\|\leq\|S^{-1}\|(\|g(T_{0})\|+\|g(T_{1})\|)\|S\|.$$
Since $(M^{*}_{z},\mathcal{H}_{K^{(\lambda_{i})}})$ is homogeneous, $i=0,1$, it is $M\ddot{o}bius$ bounded, so is $(M^{*}_{z},\mathcal{H}_{KK^{(\lambda_{i})}})$ by Lemma \ref{2310.241}.
That is, for any $g\in M\ddot{o}b$, $\|g(T_{i})\|$ is uniformly bounded, $i=0,1$.
The invertibility of $S$ means that $T$ is $M\ddot{o}bius$ bounded.
Thus, to complete the proof, it remains to show that $T$ is not similar to a direct sum of any homogeneous operator.
That is to show that $T_{i}$ is not similar to any homogeneous operator, $i=0,1$.

Assume that $T_{i}$ is similar to a homogeneous operator, say $\tilde{T}$. Since $T_{i}\in\mathcal{B}_{1}(\mathbb{D})$ and the class $\mathcal{B}_{1}(\mathbb{D})$ is closed under similarity, $\tilde{T}$ belongs to $\mathcal{B}_{1}(\mathbb{D})$. Upto unitary equivalence, every homogeneous operator in $\mathcal{B}_{1}(\mathbb{D})$ is of the form $(M^{*}_{z},\mathcal{H}_{K^{(\mu)}})$ for some $\mu>0$.
Consequently, $T_{i}\sim_sM_{z}^{*}$ on $\mathcal{H}_{K^{(\mu)}}$. Since $\lim\limits_{w\rightarrow\partial\mathbb{D}}\frac{K(w,w)K^{(\lambda_{i})}(w,w)}{K^{(\mu)}(w,w)}=0$ or $\infty$, by Lemma 3.5 of \cite{n}, we obtain that there is no non-zero bounded operator $X_{0}$ such that $X_{0}T_{i}=M^{*}_{z}X_{0}$ or $X_{0}M^{*}_{z}=T_{i}X_{0}$.
This is a contradiction.
Hence the operator $T_{i}$ is not similar to any homogeneous operator. In other words, $T$ is not similar to a direct sum of any homogeneous operator. This completes the proof.
\end{proof}
We know that for the operators in $\mathcal{B}_{1}(\mathbb{D})$, the ratio of their corresponding reproducing kernels is bounded and bounded from zero, which is a necessary condition for their similarity.
D.N. Clark and G. Misra found a class of operators for which this necessary condition is a necessary and sufficient condition for their similarity to homogeneous operators, since for $\gamma\geq1$, $(M_z,\mathcal{H}_{K_{(\gamma-1)}})\sim_s(M_z,\mathcal{H}_{K^{(\gamma)}})$.
\begin{lem}\label{2018}\cite{s}
Let $S\sim_u(M_z^*,\mathcal{H}_{K_{(\gamma)}})$ for $\gamma>0$ and $T\sim_u(M_z^*,\mathcal{H}_{K})$ with $K(z,w)=\sum\limits_{n=0}^\infty a_nz^n\bar{w}^n$, $0<a_n\leq a_{n+1}$. Then $T\sim_sS$ if and only if $\frac{K(w,w)}{K_{(\gamma)}(w,w)}$ is bounded and bounded from zero.
\end{lem}

\begin{prop}
Let $T=\left(\begin{smallmatrix}T_{0}&XT_{1}-T_{0}X\\0&T_{1}\\\end{smallmatrix}\right)\in N\mathcal{FB}_{2}(\mathbb{D})$,
where  $T_{0}\sim_{u}(M_{z}^{*},\mathcal{H}_{K^{(\mu)}})$, $T_{1}\sim_{u}(M_{z}^{*},\mathcal{H}_{K_{1}})$, $K_{1}(z,w)=\sum\limits_{n=0}^\infty a_{n}z^{n}\overline{w}^{n}$ for $0<a_{n}\leq a_{n+1}$.
If there exist a metric $h_{T}$ of $E_{T}$, $\lambda>\mu$ and constants $m,M$ such that $m\leq(\det h_{T}(w))(1-|w|^2)^{\lambda}\leq M$
if and only if $T_{1}$ is a weakly homogeneous operator similar to a homogeneous operator. 
\end{prop}

\begin{proof}
Let $\gamma_{0}(w)=\left(\begin{smallmatrix}K^{(\mu)}(\cdot,\overline{w})\\0\\\end{smallmatrix}\right)$ and $\gamma_{1}(w)=\left(\begin{smallmatrix}XK_{1}(\cdot,\overline{w})\\K_{1}(\cdot,\overline{w})\\\end{smallmatrix}\right)$.
Then $\{\gamma_{0},\gamma_{1}\}$ be a holomorphic frame of the Hermitian holomorphic vector bundle $E_T$ corresponding to $T$.
It implies that the metric matrix of $E_{T}$ under this frame is
$h_{T}(w)
=\left(\begin{smallmatrix}K^{(\mu)}(\overline{w},\overline{w})&\langle XK_{1}(\cdot,\overline{w}),K^{(\mu)}(\cdot,\overline{w})\rangle\\\langle K^{(\mu)}(\cdot,\overline{w}),XK_{1}(\cdot,\overline{w})\rangle&\|XK_{1}(\cdot,\overline{w})\|^{2}+K_{1}(\overline{w},\overline{w})\\\end{smallmatrix}\right)$.
Thus,   $det\,h_{T}(w)=K^{(\mu)}(\overline{w},\overline{w})(K_{1}(\overline{w},\overline{w})+\|XK_{1}(\cdot,\overline{w})\|^{2})-|\langle XK_{1}(\cdot,\overline{w}),K^{(\mu)}(\cdot,\overline{w})\rangle|^{2}. $ According to the Cauchy-Bunyakowsky-Schwarz Inequality, we have
$1\leq\frac{det\,h_{T}(w)}{K^{(\mu)}(\overline{w},\overline{w})K_{1}(\overline{w},\overline{w})}\leq1+\|X\|^{2}$.
It follows from the inequality in the proposition that
$\frac{m}{1+\|X\|^{2}}\leq K^{(\mu)}(\overline{w},\overline{w})K_{1}(\overline{w},\overline{w})(1-|w|^2)^{\lambda}\leq M$.
Thus, $\frac{m}{1+\|X\|^{2}}\leq \frac{K_{1}(\overline{w},\overline{w})}{K^{(\lambda-\mu)}(\overline{w},\overline{w})} \leq M.$
By Lemma \ref{2018}, we have $(M^{*}_{z},\mathcal{H}_{K_{1}})\sim_{s}(M^{*}_{z},\mathcal{H}_{K^{(\lambda-\mu)}})$.
Since $\lambda>\mu$, $(M^{*}_{z},\mathcal{H}_{K^{(\lambda-\mu)}})$ is a homogeneous operator and $(M^{*}_{z},\mathcal{H}_{K_{1}})$ is a weakly homogeneous operator.

The proof of sufficiency is similar. This completes the proof.
\end{proof}

\subsection{Irreducible weakly homogeneous operator}
At the beginning of this section, we mentioned that the operator $\left(\begin{smallmatrix}T_{0}&XT_{1}-T_{0}X\\0&T_{1}\\\end{smallmatrix}\right)$ is strongly reducible. 
The following proposition shows that irreducible operators of this form exist.
\begin{prop}
Let $T_{0}\sim_u(M^{*}_{z},\mathcal{H}_{K})\in\mathcal{B}_{1}(\mathbb{D})$ be a weakly homogeneous operator and $T_{1}\sim_u(M^{*}_{z},\mathcal{H}_{K^{(\lambda)}K})$ for $\lambda>0$.
If there exists $X$ such that $XT_{1}\neq T_{0}X$, then
$\left(\begin{smallmatrix}T_{0}&XT_{1}-T_{0}X\\0&T_{1}\\\end{smallmatrix}\right)$ belongs to $N\mathcal{FB}_{2}(\mathbb{D})$ and is an irreducible weakly homogeneous operator.
\end{prop}

\begin{proof}
We know that $T_1$ is a weakly homogeneous operator by using Theorem \ref{220509.10}.
Since
$\lim\limits_{w\to\partial\mathbb{D}}\frac{K(w,w)}{K^{(\lambda)}(w,w)K(w,w)}=\lim\limits_{w\to\partial\mathbb{D}}\frac{1}{K^{(\lambda)}(w,w)}=\lim\limits_{w\to\partial\mathbb{D}}(1-|w|^{2})^{\lambda}=0$,
by Lemma 3.5 in \cite{n}, then there is no non-zero bounded intertwining operator $X_{0}$ such that $T_{1}X_{0}=X_{0}T_{0}$.
Since $XT_{1}\neq T_{0}X$, $\left(\begin{smallmatrix}T_{0}&XT_{1}-T_{0}X\\0&T_{1}\\\end{smallmatrix}\right)\in N\mathcal{FB}_{2}(\mathbb{D})$ is weakly homogeneous.
Next, we will check the irreducibility.

Let $P=\left(\begin{smallmatrix}P_{00}&P_{01}\\P_{10}&P_{11}\\\end{smallmatrix}\right)\in\mathcal{A}'(T)$ and $P=P^{*}=P^{2}$. We have
$P_{00}=P^{*}_{00}, P_{01}=P^{*}_{10},P_{11}=P^{*}_{11}$ and
\begin{equation}\label{220509.12}
\left(\begin{smallmatrix}P_{00}T_{0}&P_{00}(XT_{1}-T_{0}X)+P_{01}T_{1}\\P_{10}T_{0}&P_{10}(XT_{1}-T_{0}X)+P_{11}T_{1}\\\end{smallmatrix}\right)
=\left(\begin{smallmatrix}T_{0}P_{00}+(XT_{1}-T_{0}X)P_{10}&T_{0}P_{01}+(XT_{1}-T_{0}X)P_{11}\\T_{1}P_{10}&T_{1}P_{11}\\\end{smallmatrix}\right).
\end{equation}
From the above proof, it can be seen that $P_{01}=0$, $P_{10}=0$. It follows that $P_{ii}\in\mathcal{A}'(T_{i})$ and $P_{ii}=P_{ii}^{*}=P_{ii}^{2}$, $i=0,1$.
Since $T_0,T_1\in\mathcal{B}_{1}(\mathbb{D})$ are irreducible operators, $P_{ii}=0$ or $I$.
For the $(1,2)$-entry of (\ref{220509.12}), we have $P_{00}=P_{11}$. We infer $P=0$ or $I$, which means $T$ is an irreducible weakly homogeneous operator.
\end{proof}

The next two irreducible results are about bilateral block shifts by Theorem \ref{220507.2},\ref{220508.6} and Lemma \ref{220509.9}.

\begin{prop}
Let $T=T(\sqrt{\frac{n+a}{n+b}},d_n)$.
If $|\frac{d_{n+1}}{t_{n}}-d_{n}t_{n}|^{2}=|\frac{d_{m+1}}{t_{m}}-d_{m}t_{m}|^{2}$ only holds when $m=-(n+1)$ and $t_{n}\overline{c_{n}}(|t_{n-1}|^{2}+\frac{1}{|t_{n-1}|^{2}}+|c_{n-1}|^{2})\neq\frac{1}{t_{n-1}}\overline{c_{n-1}}(|t_{n}|^{2}-\frac{1}{|t_{n}|^{2}}-|c_{n}|^{2})$
for $c_{n}=d_{n+1}\sqrt{\frac{n+b}{n+a}}-d_{n}\sqrt{\frac{n+a}{n+b}}$,
then $T$ is irreducible weakly homogeneous.
\end{prop}

\begin{thm}
Let $T=T(w_n,v_n,d_n)$ with $|w_{n}|=|v_{n}|=1$. Suppose that there exists a bijective mapping $g$ on $\mathbb{Z}$ and $N\in\mathbb{Z}$ such that
the sets of singular values of $A_n$ are distinct from each other.
If $|d_{n+1}v_{n}-d_{n}w_{n}|^{2}=|d_{g(n)+1}v_{g(n)}-d_{g(n)}w_{g(n)}|^{2}$ and one of the following two conditions is true,
 \begin{enumerate}
 \item [(1)]
$d_{n+1}c_{n}+d_{n}c_{n-1}-d_{n+1}\overline{w_{n}}v_{n}c_{n-1}\notin\mathbb{R}$, where $c_{n}=\overline{d_{n}}\overline{w_{n}}v_{n}$;
 \item [(2)]
$(\overline{d_{n+1}}w_{n}v_{n}-\overline{d_{n}})|d_{n}v_{n-1}-d_{n-1}w_{n-1}|^{2}\neq(\overline{d_{n-1}}\overline{w_{n-1}}v_{n-1}-\overline{d_{n}})|d_{n+1}v_{n}-d_{n}w_{n}|^{2}$,
\end{enumerate}
then $T$ is irreducible irreducible.
\end{thm}

\end{document}